\documentclass[12pt,a4paper]{article}
\usepackage{indentfirst}
\usepackage{amsfonts}
\usepackage{amssymb}
\usepackage{mathrsfs}
\usepackage{amsmath}
\usepackage{amsthm}
\usepackage{enumerate}
\usepackage{cite}
\usepackage{cases}
\allowdisplaybreaks
\newtheorem{defn}{Definition}[section]
\newtheorem{thm}[defn]{Theorem}
\newtheorem{lem}[defn]{Lemma}
\newtheorem{prop}[defn]{Proposition}
\newtheorem{cor}[defn]{Corollary}
\newtheorem{ex}[defn]{Example}
\newtheorem{re}[defn]{Remark}
\newtheorem{no}[defn]{Notation}
\begin{document}
\title{{\bf Algebras of quotients of Hom-Lie algebras}}
\author{\normalsize \bf Chenrui Yao,  Liangyun Chen}
\date{{\small{ School of Mathematics and Statistics,  Northeast Normal University,\\ Changchun 130024, CHINA
}}} \maketitle
\date{}

   {\bf\begin{center}{Abstract}\end{center}}

In this paper, we introduce the notion of algebras of quotients of Hom-Lie algebras and investigate some properties which can be lifted from a Hom-Lie algebra to its algebra of quotients. We also give some necessary and sufficient conditions for Hom-Lie algebras having algebras of quotients. We also examine the relationship between a Hom-Lie algebra and the associative algebra generated by inner derivations of the corresponding Hom-Lie algebra of quotients. Moreover, we introduce the notion of dense extensions and get a proposition about Hom-Lie algebras of quotients via dense extensions.

\renewcommand{\thefootnote}{\fnsymbol{footnote}}
\footnote[0]{ Corresponding author(L. Chen): chenly640@nenu.edu.cn.}
\footnote[0]{Supported by  NNSF of China (No. 11771069).}

\section{Introduction}
Algebras where the identities defining the structure are twisted by a homomorphism are called Hom-algebras. The notion of Hom-Lie algebras was introduced by Hartwig, Larsson and Silvestrov to describe the structures on certain deformations of the Witt algebras and the Virasoro algebras in \cite{HLS}. Hom-Lie algebras are also related to deformed vector fields, the various versions of the Yang-Baxter equations, braid group representations, and quantum groups in \cite{HLS, Y1, Y2}. More research on Hom-Lie algebras see references \cite{G1, S2}.

The notion of ring of quotients was introduced by Utumi in \cite{U1}, which played an important role in the development of the theories of associative and commutative rings. He showed that every associative ring without right zero divisors has a maximal left quotients ring and constructed it. Inspired by \cite{U1}, Siles Molina studied the algebras of quotients of Lie algebras in \cite{S1}. Martindale rings of quotients were introduced by Martindale in 1969 for prime rings in \cite{M1}, which was designed for applications to rings satisfying a generalized polynomial identity. In \cite{GG}, Garc\'{i}a and G\'{o}mez defined Martindale-like quotients for Lie triple systems with respect to power filters of sturdy ideals and constructed the maximal system of quotients in the nondegenerate cases. More research about quotients of Lie systems refer in references \cite{GC, MCL}. In \cite{M2}, Mart\'{i}nez derived a necessary and sufficient Ore type condition for a Jordan algebra to have a ring of fractions, which is the origin of algebras of quotients of Jordan systems. In \cite{AGG, AM, GG2, M3}, authors introduced and researched algebras of quotients and Martindale-like quotients of Jordan systems respectively. Inspired by \cite{S1}, we will introduce the notion of algebras of quotients of Hom-Lie algebras and prove the properties, such as semiprimeness and primeness,  can be lifted from a Hom-Lie algebra to its algebra of quotients in this paper.

As we know, one popular topic of research is the relationship between algebras of quotients of non-associative algebras and associative algebras in recent years. In \cite{PS}, authors examine the relationship between associative and Lie algebras of quotients. Inspired by them, we'll explore whether there exists a relationship between associative and Hom-Lie algebras of quotients. The answer is affirmative and we will describe in detail in the following.

The paper is organised as follows: in Section \ref{se:2}, we'll introduce some basic definitions and prove some elementary propositions which will be used in the following. In Section \ref{se:3}, we'll give the definition of algebra of quotients of a Hom-Lie algebra at first. Next, we'll introduce the concept of ideally absorbed of a Hom-Lie algebra and show that a Hom-Lie algebra $(Q, [\cdot, \cdot], \alpha)$ is an algebra of quotients of a Hom-Lie algebra $(L, [\cdot, \cdot], \alpha)$ if and only if $(Q, [\cdot, \cdot], \alpha)$ is ideally absorbed into $(L, [\cdot, \cdot], \alpha)$, presented as Theorem \ref{thm:3.10}. In Section \ref{se:4}, we focus on semiprime Hom-Lie algebras and construct the maximal algebras of quotients of them and characterize them. In Section \ref{se:5}, we mainly study the relationship between algebras of quotients of Hom-Lie algebras and associative algebras generated by inner derivations of Hom-Lie algebras, and get the main result that $A(Q)$, the associative algebra generated by inner derivations of $(Q, [\cdot, \cdot], \alpha)$, is a left quotients of $A_{0}$ if $(Q, [\cdot, \cdot], \alpha)$ is an algebra of quotients of $(L, [\cdot, \cdot], \alpha)$, where $A_{0}$ denotes the subalgebra of $A(Q)$ satisfying $A_{0}(L) \subseteq L$, written as Theorem \ref{thm:5.11}. In Section \ref{se:6}, we introduce the concept of dense extension first of all. Next, we get a proposition about dense extension, Proposition \ref{prop:6.7}, that every $I \subseteq Q$ is also dense for every essential Hom-ideal $I$ of $(L, [\cdot, \cdot], \alpha)$ if $L \subseteq Q$ is a dense extension and $(Q, [\cdot, \cdot], \alpha)$ is an algebra of quotients of $(L, [\cdot, \cdot], \alpha)$.
\section{Preliminaries}\label{se:2}
\begin{defn}\cite{HLS}
Let $L$ be a vector space over a field $\mathbb{F}$. Define a bilinear map $[\cdot, \cdot] : L \times L \rightarrow L$ and a linear map $\alpha : L \rightarrow L$. If the triple $(L, [\cdot, \cdot], \alpha)$ satisfies the following conditions:
\begin{enumerate}[(1)]
\item $[x, y] = -[y, x]$, $\forall x, y \in L$;
\item $[\alpha(x), [y, z]] + [\alpha(y), [z, x]] + [\alpha(z), [x, y]] = 0$, $\forall x, y, z \in L$.
\end{enumerate}
then $(L, [\cdot, \cdot], \alpha)$ is called a Hom-Lie algebra.
\end{defn}
\begin{defn}\cite{HLS}
A subspace $W$ of a Hom-Lie algebra $(L, [\cdot, \cdot], \alpha)$ is called a Hom-subalgebra of $L$ if it satisfies $\alpha(W) \subseteq W$ and $[W, W] \subseteq W$.
\end{defn}
\begin{defn}\cite{HLS}
A subspace $I$ of a Hom-Lie algebra $(L, [\cdot, \cdot], \alpha)$ is called a Hom-ideal of $L$ if it satisfies $\alpha(I) \subseteq I$ and $[I, L] \subseteq I$.
\end{defn}
\begin{re}\label{re:2.4}
Suppose that $(L, [\cdot, \cdot], \alpha)$ is a Hom-Lie algebra. For any $x \in L$, define $\rm{ad}_{\it{x}} : \it{L \rightarrow L}$ to be a linear map by $\rm{ad}_{\it{x}}\it{(y)} = [\alpha(x), y]$ for any $y \in L$. Then $\rm{ad}_{\it{x}}$ is a derivation if $\alpha$ satisfies
\begin{equation}\label{eq:2.1}
\alpha([x, y]) = [\alpha(x), y] = [x, \alpha(y)],\quad\forall x, y \in L,\tag{2.1}
\end{equation}
which is to say that $\alpha$ belongs to $\alpha^{0}$-centroid. In this sense, $\rm{ad}_{\it{x}}$ is called an inner derivation.
\end{re}
\begin{ex}
Suppose that $L$ is a $4$-dimensional vector over $\mathbb{F}$ with a basis $\{e_{1}, e_{2}, e_{3}, e_{4}\}$. The multiplication is
\[[e_{2}, e_{3}] = [e_{2}, e_{4}] = e_{2},\; [e_{3}, e_{4}]= e_{3},\; \rm{others\; are\; all\; zero}.\]
One can see that $(L, [\cdot, \cdot])$ isn't a Lie algebra since
\[[[e_{2}, e_{3}], e_{4}] + [[e_{3}, e_{4}], e_{2}] + [[e_{4}, e_{2}], e_{3}] = -e_{2} \neq 0.\]
Define $\alpha : L \rightarrow L$ to be a linear map by
\[\alpha(e_{1}) = ae_{1}, \alpha(e_{2}) = \alpha(e_{3}) = 0, \alpha(e_{4}) = be_{1},\]
where $a, b \in \mathbb{F}$. Then $(L, [\cdot, \cdot], \alpha)$ is a non-trivial Hom-Lie algebra satisfying (\ref{eq:2.1}).
\end{ex}
\begin{defn}
For a Hom-subalgebra $H$ of a Hom-Lie algebra $(L, [\cdot, \cdot], \alpha)$, define
\[\rm{Ann}\it{_{L}(H) = \{x \in L | [x, \alpha(y)] = \rm{0},\;\forall \it{y} \in H\}}\]
to be the $\alpha$-annihilator of $H$ in $L$. If $H$ is a Hom-ideal of $L$ and $\alpha$ satisfies (\ref{eq:2.1}), then so is $\rm{Ann}\it{_{L}(H)}$. In particular, if $H = L$, write $\rm{Ann}\it{(L)} = \rm{Ann}\it{_{L}(L)}$.
\end{defn}
\begin{defn}
Let $(L, [\cdot, \cdot], \alpha)$ be a Hom-Lie algebra. An element $x \in L$ is an absolute zero divisor if $(\rm{ad}_{\it{x}})^{2} = \rm{0}$. The Hom-Lie algebra $L$ is said to be nondegenerate if it doesn't contain any nonzero absolute zero divisor.
\end{defn}
\begin{defn}
Suppose that $(L, [\cdot, \cdot], \alpha)$ is a Hom-Lie algebra. Then $L$ is said to be semiprime if for any nonzero Hom-ideal $I$ of $L$, $[I, \alpha(I)] \neq \{0\}$ and prime if for any two nonzero Hom-ideals $I$, $J$ of $L$, $[I, \alpha(J)] \neq \{0\}$.
\end{defn}

In the following, we suppose that all Hom-Lie algebras satisfy $(\ref{eq:2.1})$.
\begin{re}\label{re:2.9}
For a Hom-Lie algebra $(L, [\cdot, \cdot], \alpha)$, we have

$L$ nondegenerate $\Rightarrow$ $L$ semiprime $\Rightarrow$ $\rm{Ann}\it{(L)} = \{\rm{0}\}$.
\end{re}
\begin{proof}
(1). Suppose that $L$ is nondegenerate. If $L$ isn't semiprime, there exists a nonzero Hom-ideal $I$ of $L$ satisfying $[I, \alpha(I)] = \{0\}$. Take $0 \neq x \in I$, then we have $(\rm{ad}_{\it{x}})^{2}(\it{L}) = [\alpha(x), [\alpha(x), L]] \subseteq [\alpha(I), I] = \{\rm{0}\}$, which implies that $x$ is a nonzero absolute zero divisor. Contradict with $L$ is nondegenerate.

(2). Suppose that $L$ is semiprime. If $\rm{Ann}\it{(L)} \neq \{\rm{0}\}$, then $\rm{Ann}\it{(L)}$ is a nonzero Hom-ideal of $L$ satisfying $[\rm{Ann}\it{(L)}, \alpha(\rm{Ann}\it{(L)})]\subseteq [\rm{Ann}\it{(L)}, \alpha(L)] = \{\rm{0}\}$. Contradiction.
\end{proof}

A Hom-ideal $I$ of a Hom-Lie algebra $(L, [\cdot, \cdot], \alpha)$ is said to be essential if every nonzero Hom-ideal of $L$ hits $I$($I \cap J \neq \{0\}$ for every nonzero Hom-ideal $J$ of $L$).
\begin{prop}\label{prop:2.10}
Suppose that $(L, [\cdot, \cdot], \alpha)$ is a Hom-Lie algebra and $I$ a Hom-ideal of $L$.

(1) If $\rm{Ann}\it{_{L}(I)} = \{\rm{0}\}$, then $I$ is essential.

(2) If $L$ is semiprime, then $I \cap \rm{Ann}\it{_{L}(I)} = \{\rm{0}\}$ and $I$ is essential if and only if $\rm{Ann}\it{_{L}(I)} = \{\rm{0}\}$.

\end{prop}
\begin{proof}
(1). Otherwise, there exists a nonzero Hom-ideal $J$ of $L$ such that $I \cap J = \{0\}$. Take $0 \neq v_{0} \in J$, we have $[v_{0}, \alpha(I)] \subseteq I \cap J = \{0\}$. So $v_{0} \in \rm{Ann}\it{_{L}(I)}$, which is a contradiction. Hence $I$ is essential.

(2). If not, $[I \cap \rm{Ann}\it{_{L}(I)}, \alpha(I \cap \rm{Ann}\it{_{L}(I)})] \subseteq [\rm{Ann}\it{_{L}(I)}, \alpha(I)] = \{\rm{0}\}$. Note that $I \cap \rm{Ann}\it{_{L}(I)}$ is a nonzero Hom-ideal of $L$, this contradicts with $L$ being semiprime.

Suppose that $I$ is essential. If $\rm{Ann}\it{_{L}(I)} \neq \{\rm{0}\}$, then $I \cap \rm{Ann}\it{_{L}(I)} \neq \{\rm{0}\}$. Contradiction. Hence, $\rm{Ann}\it{_{L}(I)} = \{\rm{0}\}$. Combining with (1), we get that $I$ is essential if and only if $\rm{Ann}\it{_{L}(I)} = \{\rm{0}\}$.
\end{proof}
\begin{prop}\label{prop:2.11}
Suppose that $(L, [\cdot, \cdot], \alpha)$ is a Hom-Lie algebra. Then $L$ is prime if and only if for every nonzero Hom-ideal $I$ of $L$, $\rm{Ann}\it{_{L}(I)} = \{\rm{0}\}$.
\end{prop}
\begin{proof}
Suppose that $L$ is prime. If there exists a nonzero Hom-ideal $J$ of $L$ such that $\rm{Ann}\it{_{L}(J)} \neq \{\rm{0}\}$. Then $[\rm{Ann}\it{_{L}(J)}, \alpha(J)] = \{\rm{0}\}$ according to the definition of $\rm{Ann}\it{_{L}(J)}$, which contradicts with $L$ being prime.

On the other hand, suppose that for every nonzero Hom-ideal $I$ of $L$, $\rm{Ann}\it{_{L}(I)} = \{\rm{0}\}$. If $L$ isn't prime, there exist two nonzero Hom-ideals $I$ and $J$ of $L$ such that $[I, \alpha(J)] = \{0\}$, which implies that $I \subseteq \rm{Ann}\it{_{L}(J)}$. Contradiction.
\end{proof}
\section{Algebras of quotients of Hom-Lie algebras}\label{se:3}
Inspired by the notion of algebras of quotients of Lie algebras in \cite{S1}, we introduce the notion of algebras of quotients of Hom-Lie algebras.

Suppose that $(L, [\cdot, \cdot], \alpha)$ is a Hom-subalgebra of Hom-Lie algebra $(Q, [\cdot, \cdot], \alpha)$. For every $q \in Q$, set
\[_{L}(q) = \mathbb{F}q + \left\{\sum_{i = 1}^{n}[[[[q, x_{1}^{i}], x_{2}^{i}], \cdots], x_{k_{i}}^{i}] \mid x_{j}^{i} \in L\; \rm{with}\; \it{n}, k_{i} \in \mathbb{N}\right\},\]

That is, $_{L}(q)$ is the linear span in $Q$ of the elements of the form $[[[[q, x_{1}], x_{2}], \cdots], x_{n}]$ and $q$, where $n \in \mathbb{N}$ and $x_{1}, \cdots, x_{n} \in L$. It is clear that $[_{L}(q), L] \subseteq _{L}(q)$. Now we define
\[(L : q) = \{x \in L \mid [x, \alpha(_{L}(q))] \subseteq L\}.\]
Obviously, if $q \in L$, then $(L : q) = L$.
\begin{prop}\label{prop:3.1}
Suppose that $(L, [\cdot, \cdot], \alpha)$ is a Hom-subalgebra of $(Q, [\cdot, \cdot], \alpha)$. Take $q \in Q$. Then $(L : q)$ is a Hom-ideal of $L$. Moreover, it's maximal among the Hom-ideals $I$ of $L$ such that $[I, \alpha(q)] \subseteq L$.
\end{prop}
\begin{proof}
For any $x \in (L : q)$, $[\alpha(x), \alpha(_{L}(q))] = \alpha([x, \alpha(_{L}(q))]) \subseteq \alpha(L) \subseteq L$, which implies that $\alpha(x) \in (L : q)$, i.e., $\alpha((L : q)) \subseteq (L : q)$.

For any $x \in (L : q)$, $y \in L$, $[[x, y], \alpha(_{L}(q))] = [\alpha(x), [y, _{L}(q)]] + [[x, _{L}(q)], \alpha(y)] = [x, \alpha([y, _{L}(q)])] + [[x, \alpha(_{L}(q))], y] \subseteq [x, \alpha(_{L}(q))] + [L, y] \subseteq L$, which implies that $[x, y] \in (L : q)$, i.e., $[(L : q), L] \subseteq (L : q)$. Hence, $(L : q)$ is a Hom-ideal of $L$.

Suppose that $I$ is a Hom-ideal of $L$ such that $[I, \alpha(q)] \subseteq L$. We'll show that for any $x_{1}, \cdots, x_{n} \in L$ and $n \in \mathbb{N}$, $[I, \alpha([[[[q, x_{1}], x_{2}], \cdots], x_{n}])] \subseteq L$. We'll prove it by induction on $n$.

When $n = 1$, for any $x_{1} \in L$, $[I, \alpha([q, x_{1}])] = [\alpha(I), [q, x_{1}]] = [\alpha(q), [I, x_{1}]] + [\alpha(x_{1}), [q, I]] = [\alpha(q), [I, x_{1}]] + [x_{1}, [\alpha(q), I]] \subseteq [\alpha(q), I] + [x_{1}, L] \subseteq L$.

Suppose that when $n = k$, $[I, \alpha([[[[q, x_{1}], x_{2}], \cdots], x_{k}])] \subseteq L$ for any $x_{1}, \cdots, x_{k} \in L$.

When $n = k + 1$, for any $x_{1}, \cdots, x_{k}, x_{k + 1} \in L$, we denote $[[[[q, x_{1}], x_{2}], \cdots], x_{k}]$ by $z$. By hypothesis, we have $[I, \alpha(z)] \subseteq L$.

Therefore, $[I, \alpha([[[[[q, x_{1}], x_{2}], \cdots], x_{k}], x_{k + 1}])] = [I, \alpha([z, x_{k + 1}])] = [\alpha(I), [z, x_{k + 1}]] = [\alpha(z), [I, x_{k + 1}]] + [\alpha(x_{k + 1}), [z, I]] = [\alpha(z), [I, x_{k + 1}]] + [x_{k + 1}, [\alpha(z), I]] \subseteq [\alpha(z), I] + [x_{k + 1}, L] \subseteq L$.

Hence, we have $[I, \alpha(_{L}(q))] \subseteq L$, i.e., $I \subseteq (L : q)$.
\end{proof}
\begin{defn}
Suppose that $(L, [\cdot, \cdot], \alpha)$ is a Hom-subalgebra of $(Q, [\cdot, \cdot], \alpha)$. Then $Q$ is called an algebra of quotients of $L$, if given $p, q \in Q$ with $p \neq 0$, there exists $x \in L$ such that $[x, \alpha(p)] \neq 0$ and $x \in (L : q)$.
\end{defn}
\begin{prop}\label{prop:3.3}
Let $(L, [\cdot, \cdot], \alpha)$ be a Hom-subalgebra of $(Q, [\cdot, \cdot], \alpha)$.
\begin{enumerate}[(1)]
\item If $\rm{Ann}\it{(L)} = \{\rm{0}\}$, then $L$ is an algebra of quotients of itself.
\item If $Q$ is an algebra of quotients of $L$, then $\rm{Ann}\it{_{Q}(L)} = \rm{Ann}\it{(L)} = \{\rm{0}\}$.
\end{enumerate}
\end{prop}
\begin{proof}
(1). For any $p, q \in L$ with $p \neq 0$, we have $p \notin \rm{Ann}\it{(L)}$. There exists $x \in L$ such that $[p, \alpha(x)] \neq 0$, so $[x, \alpha(p)] \neq 0$. It's obvious that $x \in (L : q)$. Therefore $L$ is an algebra of quotients of itself.

(2). For any $0 \neq q \in Q$, there exists $x \in L$ such that $[x, \alpha(q)] \neq 0$, so $[q, \alpha(x)] \neq 0$. Hence $q \notin \rm{Ann}\it{_{Q}(L)}$. Therefore $\rm{Ann}\it{_{Q}(L)} = \{\rm{0}\}$. Similarly, we have $\rm{Ann}\it{(L)} = \{\rm{0}\}$.
\end{proof}

The above proposition says for a Hom-Lie algebra $(L, [\cdot, \cdot], \alpha)$, that $\rm{Ann}\it{(L)} = \{\rm{0}\}$ is a sufficient and necessary condition such that $L$ has algebras of quotients.

We will prove that some properties of a Hom-Lie algebra $(L, [\cdot, \cdot], \alpha)$ can be inherited by its algebras of quotients $Q$. Actually, $Q$ just needs a weaker condition.
\begin{defn}
Suppose that $(L, [\cdot, \cdot], \alpha)$ is a Hom-subalgebra of $(Q, [\cdot, \cdot], \alpha)$. Then $Q$ is called a weak algebra of quotients of $L$, if given $0 \neq q \in Q$, there exists $x \in L$ such that $0 \neq [x, \alpha(q)] \in L$.
\end{defn}
\begin{re}
Every algebra of quotients of a Hom-Lie algebra $(L, [\cdot, \cdot], \alpha)$ is a weak algebra of quotients.
\end{re}
\begin{prop}\label{prop:3.6}
Let $(Q, [\cdot, \cdot], \alpha)$ be a weak algebra of quotients of the Hom-Lie algebra $(L, [\cdot, \cdot], \alpha)$.
\begin{enumerate}[(1)]
\item If $I$ is a nonzero Hom-ideal of $Q$, then $I \cap L$ is a nonzero Hom-ideal of $L$.
\item If $L$ is semiprime(prime), so is $Q$.
\end{enumerate}
\end{prop}
\begin{proof}
(1). Since $I$ is a Hom-ideal of $Q$, $\alpha(I) \subseteq I$. Hence $\alpha(I \cap L) \subseteq \alpha(I) \cap \alpha(L) \subseteq I \cap L$.

Note that $I \neq \{0\}$, we can take $0 \neq x_{0} \in I$. There exists $x \in L$ such that $0 \neq [x, \alpha(x_{0})] \in L$. $[x, \alpha(x_{0})] \in I$ since $I$ is a Hom-ideal of $Q$. So we have $0 \neq [x, \alpha(x_{0})] \in I \cap L$.

Moreover, $[I \cap L, L] \subseteq [I, Q] \subseteq I$ and $[I \cap L, L] \subseteq [L, L] \subseteq L$. Hence, $[I \cap L, L] \subseteq I \cap L$. Therefore, $I \cap L$ is a nonzero Hom-ideal of $L$.

(2). Suppose that $L$ is semiprime. If $Q$ isn't semiprime, there exists a nonzero Hom-ideal $I$ of $Q$ such that $[I, \alpha(I)] = \{0\}$. By (1), $I \cap L$ is a nonzero Hom-ideal of $L$ and $[I \cap L, \alpha(I \cap L)] \subseteq [I, \alpha(I)] = \{0\}$. Contradict with $L$ being semiprime.

Similarly, we have $Q$ is prime if $L$ is prime.
\end{proof}
\begin{defn}
Let $(L, [\cdot, \cdot], \alpha)$ be a Hom-subalgebra of $(Q, [\cdot, \cdot], \alpha)$. Then $Q$ is called ideally absorbed into $L$ if, for every $0 \neq q \in Q$, there exists a Hom-ideal $I$ of $L$ with $\rm{Ann}\it{_{L}(I)} = \{\rm{0}\}$ such that $\{0\} \neq [I, \alpha(q)] \in L$.
\end{defn}
\begin{prop}\label{prop:3.8}
Suppose that $(L, [\cdot, \cdot], \alpha)$ is a Hom-subalgebra of $(Q, [\cdot, \cdot], \alpha)$. Take $q \in Q$.

(1) If $Q$ is an algebra of quotients of $L$, then $(L : q)$ is an essential Hom-ideal of $L$. Moreover, $\rm{Ann}\it{_{L}((L : q))} = \{\rm{0}\}$.

(2) If $Q$ is ideally absorbed into $L$, then $(L : q)$ is an essential Hom-ideal of $L$. Moreover, $\rm{Ann}\it{_{L}((L : q))} = \{\rm{0}\}$.

\end{prop}
\begin{proof}
(1). Let $I$ be a nonzero Hom-ideal of $L$. Take $0 \neq y \in I$. Apply that $Q$ is an algebra of quotients of $L$ to find $x \in L$ such that $0 \neq [x, \alpha(y)]$ and $x \in (L : q)$. Since $(L : q)$ is a Hom-ideal of $L$, $[x, \alpha(y)] \in I \cap (L : q)$. This implies that $I \cap (L : q) \neq \{0\}$. Therefore $(L : q)$ is an essential Hom-ideal of $L$.

Suppose that $\rm{Ann}\it{_{L}((L : q))} \neq \{\rm{0}\}$. Then $\rm{Ann}\it{_{L}((L : q))} \cap (L : q) \neq \{\rm{0}\}$ since $(L : q)$ is essential. Take $0 \neq u \in \rm{Ann}\it{_{L}((L : q))} \cap (L : q)$. Apply that $Q$ is an algebra of quotients of $L$ to find $x \in L$ such that $0 \neq [x, \alpha(u)]$ and $x \in (L : q)$. However, $[x, \alpha(u)] = -[u, \alpha(x)] = 0$ since $u \in \rm{Ann}\it{_{L}((L : q))}$. Contradiction.

(2). Since $Q$ is ideally absorbed into $L$, there exists a nonzero Hom-ideal $I$ of $L$ with $\rm{Ann}\it{_{L}(I)} = \{\rm{0}\}$ such that $\{0\} \neq [I, \alpha(q)] \subseteq L$. According to the proof of Proposition \ref{prop:3.1}, $[I, \alpha(_{L}(q))] \subseteq L$. So $I \subseteq (L : q)$. Hence, $\rm{Ann}\it{_{L}((L : q))} \subseteq \rm{Ann}\it{_{L}(I)} = \{\rm{0}\}$. According to Proposition \ref{prop:2.10} (1), $(L : q)$ is essential.
\end{proof}
\begin{lem}\label{le:3.9}
Let $(Q, [\cdot, \cdot], \alpha)$ be a weak algebra of quotients of Hom-Lie algebra $(L, [\cdot, \cdot], \alpha)$. Let $I$ be a Hom-ideal of $L$ with $\rm{Ann}\it{_{L}(I)} = \{\rm{0}\}$. Then there is no nonzero element $x$ in $Q$ such that $[\alpha(x), I] = \{0\}$.
\end{lem}
\begin{proof}
Suppose that there exists $0 \neq x \in Q$ such that $[\alpha(x), I] = \{0\}$. Apply that $Q$ is a weak algebra of quotients of $L$ to find $y \in L$ such that $0 \neq [y, \alpha(x)] \in L$. Since $\rm{Ann}\it{_{L}(I)} = \{\rm{0}\}$, $[[y, \alpha(x)], \alpha(I)] \neq \{0\}$. However,
\begin{align*}
&[[y, \alpha(x)], \alpha(I)] = \alpha([[y, x], \alpha(I)]) = \alpha([\alpha(x), [I, y]] + [\alpha(y), [x, I]])\\
&= \alpha([\alpha(x), [I, y]] + [y, [\alpha(x), I]]) = \{0\}.
\end{align*}
This is a contradiction.
\end{proof}
\begin{thm}\label{thm:3.10}
Suppose that $(L, [\cdot, \cdot], \alpha)$ is a Hom-subalgebra of $(Q, [\cdot, \cdot], \alpha)$. Then $Q$ is an algebra of quotients of $L$ if and only if $Q$ is ideally absorbed into $L$.
\end{thm}
\begin{proof}
Suppose that $Q$ is an algebra of quotients of $L$. For any $0 \neq q \in Q$, $(L : q)$ is a Hom-ideal of $L$ with $\rm{Ann}\it{_{L}((L : q))} = \{\rm{0}\}$ according to Proposition \ref{prop:3.8} (1). By Lemma \ref{le:3.9}, $[\alpha(q), (L : q)] \neq \{0\}$. According to the definition of $(L : q)$, $[(L : q), \alpha(q)] \subseteq L$. Hence, $Q$ is ideally absorbed into $L$.

Now suppose that $Q$ is ideally absorbed into $L$. Then $Q$ is a weak algebra of quotients of $L$. For any $p, q \in Q$ with $p \neq 0$, we have $\rm{Ann}\it{_{L}((L : q))} = \{\rm{0}\}$. Hence, $[\alpha(p), (L : q)] = [p, \alpha((L : q))] \neq \{0\}$. Therefore, there exists $x \in (L : q)$ such that $[\alpha(p), x] \neq 0$. Hence, $Q$ is an algebra of quotients of $L$.
\end{proof}

In the proof of Theorem \ref{thm:3.10}, we can conclude that $Q$ is an algebra of quotients of $L$ if and only if $Q$ is a weak algebra of quotients of $L$ satisfying $\rm{Ann}\it{_{L}((L : q))} = \{\rm{0}\}$ for every $q \in Q$.
\begin{prop}\label{prop:3.11}
Suppose that $(L, [\cdot, \cdot], \alpha)$ is a Hom-subalgebra of $(Q, [\cdot, \cdot], \alpha)$ and $Q$ a weak algebra of quotients of $L$. Then for every Hom-ideal $I$ of $L$, $\rm{Ann}\it{_{L}(I)} = \{\rm{0}\}$ implies $\rm{Ann}\it{_{Q}(I)} = \{\rm{0}\}$.
\end{prop}
\begin{proof}
Suppose that $0 \neq q \in \rm{Ann}\it{_{Q}(I)}$. By hypothesis, there exists an element $x$ in $L$ such that $0 \neq [x, \alpha(q)] \in L$. Since $\rm{Ann}\it{_{L}(I)} = \{\rm{0}\}$, there exists $y \in I$ such that $[[x, \alpha(q)], \alpha(y)] \neq 0$. However,
\begin{align*}
&[[x, \alpha(q)], \alpha(y)] = \alpha([[x, q], \alpha(y)]) = \alpha([[y, q], \alpha(x)] + [[x, y], \alpha(q)])\\
&= \alpha([[\alpha(y), q], x] + [\alpha([x, y]), q]) \subseteq \alpha([[\alpha(y), q], x] + [\alpha(I), q]) = 0.
\end{align*}
This is a contradiction.
\end{proof}
\begin{no}
Denote by $\mathscr{J}_{e}(L)$ the set of all essential Hom-ideals of a Hom-Lie algebra $(L, [\cdot, \cdot], \alpha)$.
\end{no}

Clearly, given $I, J \in \mathscr{J}_{e}(L)$, we have $I \cap J \in \mathscr{J}_{e}(L)$. If $L$ is semiprime, then $I^{2} \in \mathscr{J}_{e}(L)$, where $I^{2} = [I, \alpha(I)]$. If $L$ is prime, then $[I, \alpha (J)] \in \mathscr{J}_{e}(L)$ for any $I, J \in \mathscr{J}_{e}(L)$.
\begin{prop}\label{prop3.13}
Let $(L, [\cdot, \cdot], \alpha)$ be a semiprime Hom-Lie algebra and $(Q, [\cdot, \cdot], \alpha)$ an algebra of quotients of $L$. Then for every essential Hom-ideal $I$ of $L$, we have that $Q$ is an algebra of quotients of $I$.
\end{prop}
\begin{proof}
For any $p, q \in Q$ with $p \neq 0$, we have $(L : q) \in \mathscr{J}_{e}(L)$. So that $(L : q) \cap I \in \mathscr{J}_{e}(L)$ and consequently, $((L : q) \cap I)^{2} \in \mathscr{J}_{e}(L)$. Then we have $\rm{Ann}\it{_{Q}(((L : q) \cap I)^{2})} = \{\rm{0}\}$. So there exist $y, z \in (L : q) \cap I$ such that $[\alpha([y, \alpha(z)]), p] \neq 0$, i.e., $[[y, \alpha(z)], \alpha(p)] \neq 0$. Moreover,
\begin{align*}
&[[y,\alpha(z)], \alpha(_{I}(q))] = \alpha([[_{I}(q), z], \alpha(y)] + [[y, _{I}(q)], \alpha(z)]) \subseteq \alpha([[_{L}(q), z], \alpha(y)] + [[y, _{L}(q)], \alpha(z)])\\
&= \alpha([[\alpha(_{L}(q)), z], y] + [[y, \alpha(_{L}(q))], z]) \subseteq \alpha([L, I]) \subseteq \alpha(I) \subseteq I,
\end{align*}
which implies that $[y, \alpha(z)] \in (I : q)$. Therefore, $Q$ is an algebra of quotients of $I$.
\end{proof}
\section{The maximal algebra of quotients of a semiprime Hom-Lie algebra}\label{se:4}
In this section, we will construct a maximal algebra of quotients for every semiprime Hom-Lie algebra $(L, [\cdot, \cdot], \alpha)$ with $\alpha(L)$ is an essential ideal. Maximal in the sense that every algebra of quotients of $L$ can be considered inside this maximal algebra of quotients via an injective map which restricted in $L$ is the identity.
\begin{defn}
Given a Hom-ideal $I$ of a Hom-Lie algebra $(L, [\cdot, \cdot], \alpha)$, we say that a linear map $\delta : I \rightarrow L$ is an $\alpha^{k}$-partial derivation for $k \in \mathbb{N}$ if for any $x, y \in I$ it is satisfied:
\begin{enumerate}[(1)]
\item $\delta \circ \alpha = \alpha \circ \delta$;
\item $\delta([x, y]) = [\delta(x), \alpha^{k}(y)] + [\alpha^{k}(x), \delta(y)]$.
\end{enumerate}
\end{defn}
Denote by $\rm{PDer}_{\it{k}}\it{(I, L)}$ the set of all $\alpha^{k}$-partial derivations of $I$ in $L$.
\begin{lem}\label{le:4.2}
Let $(L, [\cdot, \cdot], \alpha)$ be a semiprime Hom-Lie algebra and consider the set
\[\mathscr{D}_{k} := \{(\delta, I)\;|\;I \in \mathscr{J}_{e}(L), \delta \in \rm{PDer}_{\it{k}}\it{(I, L)}\}.\]
Define on $\mathscr{D}_{k}$ the following relation:

$(\delta, I) \equiv (\mu, J)$ if and only if there exists $K \in \mathscr{J}_{e}(L)$, $K \subseteq I \cap J$ such that
\[\delta|_{K} = \mu|_{K}.\]
Then $\equiv$ is an equivalence relation.
\end{lem}
\begin{proof}
(1). For any $(\delta, I) \in \mathscr{D}_{k}$, $\delta|_{I} = \delta|_{I}$, so $(\delta, I) \equiv (\delta, I)$.

(2). For any $(\delta, I), (\mu, J) \in \mathscr{D}_{k}$ and $(\delta, I) \equiv (\mu, J)$, there exists $K \in \mathscr{J}_{e}(L)$, $K \subseteq I \cap J$ such that $\delta|_{K} = \mu|_{K}$. It's obvious that $\mu|_{K} = \delta|_{K}$. Hence $(\mu, J) \equiv (\delta, I)$.

(3). For any $(\delta, I), (\xi, J), (\beta, K) \in \mathscr{D}_{k}$ and $(\delta, I) \equiv (\xi, J)$ and $(\xi, J) \equiv (\beta, K)$, there exist $P, Q \in \mathscr{J}_{e}(L)$, $P \subseteq I \cap J$, $Q \subseteq J \cap K$ such that $\delta|_{P} = \xi|_{P}$ and $\xi|_{Q} = \beta|_{Q}$. Since $P, Q \in \mathscr{J}_{e}(L)$, $P \cap Q \neq \emptyset$. Take $T \subseteq P \cap Q$. We also get $T \in \mathscr{J}_{e}(L)$. Moreover, $\delta|_{T} = \xi|_{T} = \beta|_{T}$. Therefore $(\delta, I) \equiv (\beta, K)$.

The proof is completed.
\end{proof}
\begin{no}\label{no:4.3}
Denote by $Q_{k}(L)$ the quotient set $\mathscr{D}_{k}/\equiv$. Set $Q(L) = Q_{0}(L)$. Let $\delta_{I}$ denote the equivalence class of $(\delta, I)$ in $Q(L)$.
\end{no}
\begin{thm}\label{thm:4.4}
Let $(L, [\cdot, \cdot], \alpha)$ be a semiprime Hom-Lie algebra over $\mathbb{F}$, and let $Q := Q(L)$ be as in Notation \ref{no:4.3}. Define the following maps:
\[\cdot : \mathbb{F} \times Q \rightarrow Q,\;(p, \delta_{I}) \mapsto (p\delta)_{I}\;\rm{where}\; \it{p}\delta : I \rightarrow L,\; y \mapsto \delta(py),\]
\[+ : Q \times Q \rightarrow Q,\;(\delta_{I}, \mu_{J}) \mapsto (\delta + \mu)_{I \cap J}\;\rm{where}\;\delta + \mu : \it{I} \cap J \rightarrow L,\; x \mapsto \delta(x) + \mu(x),\]
\[[\cdot, \cdot] : Q \times Q \rightarrow Q,\; (\delta_{I}, \mu_{J}) \mapsto [\delta, \mu]_{[I \cap J, I \cap J]}\;\rm{where}\;[\delta, \mu] : \it{[I \cap J, I \cap J]} \rightarrow L,\; x \mapsto \delta\mu(x) - \mu\delta(x),\]
\[\tilde{\alpha} : Q \rightarrow Q,\; \delta_{I} \mapsto (\alpha \circ \delta)_{I}\;\rm{where}\; \alpha \circ \delta : \it{I} \rightarrow L,\; x \mapsto \alpha\delta(x).\]
Then $Q$, with these operations, is a Hom-Lie algebra satisfying
$$\tilde{\alpha}([\delta_{I}, \mu_{J}]) = [\tilde{\alpha}(\delta_{I}), \mu_{J}] = [\delta_{I}, \tilde{\alpha}(\mu_{J})]$$
for any $\delta_{I}, \mu_{J} \in Q$ and containing $L$ as a subalgebra, via the injection
\[\varphi : L \rightarrow Q,\quad x \mapsto (\rm{ad}_{\it{x}})_{\it{L}}.\]
\end{thm}
\begin{proof}
(1). For any $\delta_{I}, \mu_{J} \in Q$ and $x, y \in I \cap J$,
\begin{align*}
&[\delta, \mu]([x, y]) = \delta\mu([x, y]) - \mu\delta([x, y]) = \delta([\mu(x), y] + [x, \mu(y)]) - \mu([\delta(x), y] + [x, \delta(y)])\\
&= [\delta\mu(x), y] + [\mu(x), \delta(y)] + [\delta(x), \mu(y)] + [x, \delta\mu(y)] - [\mu\delta(x), y] - [\delta(x), \mu(y)]\\
&- [\mu(x), \delta(y)] - [x, \mu\delta(y)]\\
&= [[\delta, \mu](x), y] + [x, [\delta, \mu](y)],
\end{align*}
which implies that $[\delta, \mu] \in \rm{PDer}_{0}\it{([I \cap J, I \cap J], L)}$.

For any $\delta_{I} \in Q$ and $x, y \in I$,
\begin{align*}
&\alpha \circ \delta([x, y]) = \alpha([\delta(x), y] + [x, \delta(y)]) = [\alpha \circ \delta(x), y] + [x, \alpha \circ \delta(y)],
\end{align*}
which implies that $\alpha \circ \delta \in \rm{PDer}_{0}\it{(I, L)}$.

Therefore, $[\cdot, \cdot]$, $\tilde{\alpha}$ are well-defined.

(2). It's obvious that $(Q, \cdot, +)$ is a vector space over $\mathbb{F}$ and $[\cdot, \cdot]$ is skew-symmetric. Next we'll show that $Q$ satisfies the Jacobi identity. For any $\delta_{I}, \mu_{J}, \beta_{K} \in Q$, we have
\[[[\delta_{I}, \mu_{J}], \tilde{\alpha}(\beta_{K})] + [[\mu_{J}, \beta_{K}], \tilde{\alpha}(\delta_{I})] + [[\beta_{K}, \delta_{I}], \tilde{\alpha}(\mu_{J})] = 0\]
since $\alpha$ commutes with $\delta, \mu, \beta$. Hence, $Q$ is a Hom-Lie algebra.

(3). For any $\delta_{I}, \mu_{J} \in Q$,
\begin{align*}
&\tilde{\alpha}([\delta_{I}, \mu_{J}]) = \tilde{\alpha}([\delta, \mu]_{[I \cap J, I \cap J]}) = (\alpha \circ [\delta, \mu])_{[I \cap J, I \cap J]} = (\alpha\delta\mu - \alpha\mu\delta)_{[I \cap J, I \cap J]}\\
&= (\alpha\delta\mu - \mu\alpha\delta)_{[I \cap J, I \cap J]} = [\alpha\delta, \mu]_{[I \cap J, I \cap J]} = [(\alpha\delta)_{I}, \mu_{J}] = [\tilde{\alpha}(\delta_{I}), \mu_{J}].
\end{align*}
Similarly, we get $\tilde{\alpha}([\delta_{I}, \mu_{J}]) = [\delta_{I}, \tilde{\alpha}(\mu_{J})]$. Hence, we have $\tilde{\alpha}([\delta_{I}, \mu_{J}]) = [\tilde{\alpha}(\delta_{I}), \mu_{J}] = [\delta_{I}, \tilde{\alpha}(\mu_{J})]$ for any $\delta_{I}, \mu_{J} \in Q$.

(4). According to Remark \ref{re:2.4}, we get for any $x \in L$, $\rm{ad}_{\it{x}} \in \rm{PDer}_{0}\it{(L, L)}$. So $\varphi$ is well-defined. If $\varphi(x) = 0$, we have $(\rm{ad}_{\it{x}})_{\it{L}} = \rm{0}$. Then we have $[x, \alpha(L)] = [\alpha(x), L]= \{0\}$, which implies that $x \in \rm{Ann}\it{(L)}$. Note that $L$ is semiprime, we have $\rm{Ann}\it{(L)} = \{\rm{0}\}$ by Proposition \ref{prop:2.10} (2). Hence $x = 0$, i,e, $\varphi$ is injective.
\end{proof}

For any $X \subseteq L$, write $X^{\varphi}$ to denote the image of $X$ inside $Q(L)$ via the map defined in Theorem \ref{thm:4.4}.
\begin{lem}\label{le:4.5}
For every $\delta_{I} \in Q(L)$, and $(\rm{ad}_{\it{x}})_{\it{L}} \in \it{I}^{\varphi}(x \in I)$ we have $[\delta_{I}, (\rm{ad}_{\it{x}})_{\it{L}}] = (\rm{ad}_{\it{\delta(x)}})_{\it{L}} \in \it{L}^{\varphi}$.
\end{lem}
\begin{proof}
For any $y \in I$,
\begin{align*}
&[\delta, \rm{ad}_{\it{x}}](\it{y}) = \delta\rm{ad}_{\it{x}}(\it{y}) - \rm{ad}_{\it{x}}\delta(\it{y}) = \delta([\alpha(x), y]) - [\alpha(x), \delta(y)]\\
&= [\alpha(x), \delta(y)] + [\delta\alpha(x), y] - [\alpha(x), \delta(y)] = [\delta\alpha(x), y] = [\alpha\delta(x), y] = \rm{ad}_{\it{\delta(x)}}(\it{y}),
\end{align*}
and so $[\delta, \rm{ad}_{\it{x}}]_{\it{I}} = (\rm{ad}_{\it{\delta(x)}})_{\it{I}} = (\rm{ad}_{\it{\delta(x)}})_{\it{L}} \in \it{L}^{\varphi}$.
\end{proof}
\begin{prop}\label{prop:4.6}
Let $(L, [\cdot, \cdot], \alpha)$ be a semiprime Hom-Lie algebra with $\alpha(L)$ is essential. Then $Q(L)$ is semiprime and an algebra of quotients of $L$. Moreover, $Q(L)$ is maximal among the algebras of quotients of $L$, in the sense that if $S$ is an algebra of quotients of $L$, then there exists an injection $\psi : S \rightarrow Q(L)$ which is the identity in $L$. In particular, the map
\[\psi : S \rightarrow Q(L),\quad s \mapsto (\rm{ad}_{\it{s}})_{\it{(L : s)}}\]
is an injective map which is the identity when restricted to $L$.
\end{prop}
\begin{proof}
Take $\delta_{I}, \mu_{I} \in Q(L)$ with $\delta_{I} \neq 0$(we can consider the same $I$ for $\delta$ and $\mu$ because if $J, K \in \mathscr{J}_{e}(L)$ are such that $\delta : J \rightarrow L$, $\mu : K \rightarrow L$, then $I := J \cap K \in \mathscr{J}_{e}(L)$ and so $\delta_{J} = \delta_{I}$ and $\mu_{K} = \mu_{I}$). Choose $a \in I$ such that $\delta(a) \neq 0$. Then $(\rm{ad}_{\it{a}})_{\it{L}} \in \it{L}^{\varphi}$ satisfies $[\tilde{\alpha}(\delta_{\it{I}}), (\rm{ad}_{\it{a}})_{\it{L}}] = (\rm{ad}_{\it{\alpha\delta(a)}})_{\it{L}} \neq \rm{0}$. Otherwise, $[\delta(a), \alpha^{2}(L)] = [\alpha^{2}\delta(a), L] = \{0\}$, which implies that $\delta(a) \in \rm{Ann}\it{_{L}(\alpha(L))} = \{\rm{0}\}$. And for every $\lambda_{I} \in _{L}^{\varphi}(\mu_{I})$, $[\tilde{\alpha}(\lambda_{I}), (\rm{ad}_{\it{a}})_{\it{L}}] = (\rm{ad}_{\it{\alpha\lambda(a)}})_{\it{L}} \in \it{L}^{\varphi}$. This implies that $(\rm{ad}_{\it{a}})_{\it{L}} \in \it{(L^{\varphi} : \mu_{I})}$. Hence, $Q(L)$ is an algebra of quotients of $L^{\varphi}$. The semiprimeness of $Q(L)$ follows by Proposition \ref{prop:3.6} (2).

Now suppose that $S$ is an algebra of quotients of $L$ and consider the map
\[\psi : S \rightarrow Q(L),\quad s \mapsto (\rm{ad}_{\it{s}})_{\it{(L : s)}}.\]
According to Proposition \ref{prop:3.8} (1), $\psi$ is well-defined. To prove the injectivity, suppose that $s \in S$ such that $\psi(s) = 0$, that is $[\alpha(s), K] = \{0\}$ for some Hom-ideal $K \in \mathscr{J}_{e}(L)$, $K \subseteq (L : s)$. This implies that $s \in \rm{Ann}\it{_{L}(K)}$. Note that $K$ is essential, $\rm{Ann}\it{_{L}(K)} = \{\rm{0}\}$, so $s = 0$.
\end{proof}
\begin{defn}
For a semiprime Hom-Lie algebra $(L, [\cdot, \cdot], \alpha)$ with $\alpha(L)$ is essential, the Hom-Lie algebra constructed in Theorem \ref{thm:4.4} will be called the maximal algebra of quotients of $L$. Denote it by $Q_{m}(L)$.
\end{defn}

The axiomatic characterization of the Martindale ring of quotients given by D. Passman in \cite{P1} has inspired us to give the following description of the maximal algebra of quotients of a semiprime Hom-Lie algebra.
\begin{prop}\label{prop:4.8}
Let $(L, [\cdot, \cdot], \alpha)$ be a semiprime Hom-Lie algebra with $\alpha(L)$ is essential and consider an overalgebra $S$ of $L$. Then there exists an injection between $S$ and $Q_{m}(L)$ which is the identity on $L$, if and only if $S$ satisfies the following properties:
\begin{enumerate}[(1)]
\item For any $s \in S$, there exists $I \in \mathscr{J}_{e}(L)$ such that $[\alpha(I), s] \subseteq L$.
\item For $s \in S$ and $I \in \mathscr{J}_{e}(L)$, $[\alpha(I), s] = \{0\}$ implies that $s = 0$.
\end{enumerate}
\end{prop}
\begin{proof}
Define $\psi : S \rightarrow Q_{m}(L)$, $s \mapsto (\rm{ad}_{\it{s}})_{\it{I}}$ where $I$ is a nonzero essential Hom-ideal of $L$ satisfying $[\alpha(I), s] \subseteq L$; this exists by (1), and so the map is well-defined.

If $\psi(s) = 0$, i.e., $(\rm{ad}_{\it{s}})_{\it{I}} = \rm{0}$, then $[\alpha(s), I] = \{0\}$. By (2), we get $s = 0$. This implies that $\psi$ is an injective map.

Finally, $\psi$ is the identity on $L$, by identifying $L$ with $L^{\varphi}$, where $\varphi$ is the map defined in Theorem \ref{thm:4.4}.

Conversely, we'll prove that $Q(L)$ satisfies the two conditions.

(1). Consider $q \in Q(L)$. According to Propositions \ref{prop:4.6} and \ref{prop:3.8} (1), $(L : q) \in \mathscr{J}_{e}(L)$, and by the definition, $[\alpha((L : q)), q] = [(L : q), \alpha(q)] \subseteq L$.

(2). Take $q \in Q(L)$ and $I \in \mathscr{J}_{e}(L)$ such that $[\alpha(I), q] = \{0\}$. We have $q \in \rm{Ann}\it{_{Q(L)}(I)} = \rm{0}$ according to Propositions \ref{prop:4.6} and \ref{prop:3.11}.
\end{proof}
\section{Algebras of quotients of associative algebras generated by inner derivations of Hom-Lie algebras}\label{se:5}
In \cite{PS}, authors examined how the notion of algebras of quotients for Lie algebras tied up with the corresponding well-known concept in the associative case. Inspired by the method in \cite{PS}, we mainly study the relationship between Hom-Lie algebras and the associative algebras generated by inner derivations of the corresponding Hom-Lie algebras of quotients. First of all, we will give some definitions and basic notations.

We shall denote by $A(L)$ the associative subalgebra (possibly without identity) of $\rm{End}\it{(L)}$ generated by the elements $\rm{ad}\it{_{x}}$ for $x$ in $(L, [\cdot, \cdot], \alpha)$.

By an extension of Hom-Lie algebras $L \subseteq Q$ we will mean that $(L, [\cdot, \cdot], \alpha)$ is a Hom-subalgebra of the Hom-Lie algebra $(Q, [\cdot, \cdot], \alpha)$.

Let $L \subseteq Q$ be an extension of Hom-Lie algebras and let $A_{Q}(L)$ be the associative subalgebra of $A(Q)$ generated by $\{\rm{ad}\it{_{x}} : x \in L\}$.

Recall that, given an associative algebra $A$ and a subset $X$ of $A$, we define the right annihilator of $X$ in $A$ as
\[rann_{A}(X) = \{a \in A | Xa = 0\},\]
which is always a right ideal of $A$ (and two-sided if $X$ is a right ideal). One similarly defines the left annihilator, which shall be denoted by $lann_{A}(X)$.
\begin{lem}\label{le:5.1}
Let $I$ be a Hom-ideal of a Hom-Lie algebra $(L, [\cdot, \cdot], \alpha)$ with $\rm{Ann}\it{(L)} = \{\rm{0}\}$. Then $\rm{Ann}_{\it{L}}(\it{I}) = \{\rm{0}\}$ if and only if $rann_{A(L)}(A_{L}(I)) = \{0\}$.
\end{lem}
\begin{proof}
Suppose that $\rm{Ann}_{\it{L}}(\it{I}) = \{\rm{0}\}$. For any $\mu$ in $rann_{A(L)}(A_{L}(I))$, we have $\rm{ad}\it{_{y}}\mu = \rm{0}$ for any $y \in I$. In particular if $x \in L$, we get $0 = \rm{ad}\it{_{y}}\mu(x) = [\alpha(y), \mu(x)]$, and this implies that $\mu(x) \in \rm{Ann}_{\it{L}}(\it{I}) = \rm{0}$. Hence, $\mu = 0$.

Conversely, suppose that $rann_{A(L)}(A_{L}(I)) = \{0\}$. For any $x$ in $\rm{Ann}_{\it{L}}(\it{I})$, $y$ in $I$ and $z$ in $L$, we have
\[\rm{ad}\it{_{y}}\rm{ad}\it{_{x}}(z) = [\alpha(y), [\alpha(x), z]] = \alpha([\alpha(y), [x, z]]) = \alpha([[z, y], \alpha(x)] + [[y, x], \alpha(z)]) = \rm{0},\]
which implies that $\rm{ad}\it{_{x}} \in rann_{A(L)}(A_{L}(I)) = \rm{0}$. Since by assumption $\rm{Ann}\it{(L)} = \{\rm{0}\}$, we obtain $x = 0$.
\end{proof}
For a subset $X$ of an associative algebra $A$, denote by $\langle X \rangle_{A}^{l}$, $\langle X \rangle_{A}^{r}$ and $\langle X \rangle_{A}$ the left, right and two sided ideal of $A$, respectively, generated by $X$.
\begin{lem}\label{le:5.2}
Suppose that $(L, [\cdot, \cdot], \alpha)$ is a Hom-subalgebra of $(Q, [\cdot, \cdot], \alpha)$ and $I$ a Hom-ideal of $(L, [\cdot, \cdot], \alpha)$. Then
\[\langle A_{Q}(I) \rangle_{A_{Q}(L)}^{l} = \langle A_{Q}(I) \rangle_{A_{Q}(L)}^{r} = \langle A_{Q}(I) \rangle_{A_{Q}(L)}.\]
\end{lem}
\begin{proof}
Obviously, we have $\langle A_{Q}(I) \rangle_{A_{Q}(L)}^{l} = A_{Q}(L)A_{Q}(I) + A_{Q}(I)$ and $\langle A_{Q}(I) \rangle_{A_{Q}(L)}^{r} = A_{Q}(I)A_{Q}(L) + A_{Q}(I)$. Notice that given $x \in L$ and $y \in I$, we have
\[\rm{ad}\it{_{x}}\rm{ad}\it{_{y}} = \rm{ad}\it{_{\alpha([x, y])}} + \rm{ad}\it{_{y}}\rm{ad}\it{_{x}}.\]
Then we have $\langle A_{Q}(I) \rangle_{A_{Q}(L)}^{l} = \langle A_{Q}(I) \rangle_{A_{Q}(L)}^{r}$. Hence we come to the conclusion.
\end{proof}
\begin{lem}\label{le:5.3}
Suppose that $(L, [\cdot, \cdot], \alpha)$ is a Hom-subalgebra of $(Q, [\cdot, \cdot], \alpha)$ and $I$ a Hom-ideal of $(L, [\cdot, \cdot], \alpha)$. Write $\tilde{I}$ to denote the ideal of $A_{Q}(L)$ generated by $A_{Q}(I)$. Then
\begin{enumerate}[(1)]
\item $rann_{A_{Q}(L)}(\tilde{I}) = rann_{A_{Q}(L)}(A_{Q}(I))$.
\item $lann_{A_{Q}(L)}(\tilde{I}) = lann_{A_{Q}(L)}(A_{Q}(I))$.
\end{enumerate}
\end{lem}
\begin{proof}
(1) Since $A_{Q}(I) \subseteq \tilde{I}$, we have $rann_{A_{Q}(L)}(\tilde{I}) \subseteq rann_{A_{Q}(L)}(A_{Q}(I))$. So it's enough to show that $rann_{A_{Q}(L)}(A_{Q}(I)) \subseteq rann_{A_{Q}(L)}(\tilde{I})$. Let $\lambda \in rann_{A_{Q}(L)}(A_{Q}(I))$. By Lemma \ref{le:5.2} we know that, if $\mu \in \tilde{I}$ there exist a natural number $n$, elements $x_{1, i}, \cdots, x_{r_{i}, i} \in L$ and $y_{1, i}, \cdots, y_{s_{i}, i} \in I$ with $0 \leq r_{i} \in \mathbb{N}$ for all $i$ and $\emptyset \neq \{s_{1}, \cdots, s_{n}\} \subseteq \mathbb{N}$, such that
\[\mu = \sum_{i}^{n}\rm{ad}\it{_{x_{1, i}}} \cdots \rm{ad}\it{_{x_{r_{i}, i}}}\rm{ad}\it{_{y_{1, i}}} \cdots \rm{ad}\it{_{y_{s_{i}, i}}}.\]
Since $\rm{ad}\it{_{y_{s_{i}, i}}}\lambda = \rm{0}$, we see that $\mu\lambda = 0$.

(2) The proof is similar to (1).
\end{proof}
\begin{lem}\label{le:5.4}
Suppose that $(L, [\cdot, \cdot], \alpha)$ is a Hom-subalgebra of $(Q, [\cdot, \cdot], \alpha)$ such that $Q$ is a weak algebra of quotients of $L$. Let $I$ be a Hom-ideal of $(L, [\cdot, \cdot], \alpha)$. If $\rm{Ann}_{\it{L}}(\it{I}) = \{\rm{0}\}$, then $rann_{A(Q)}(A_{Q}(I)) = \{0\}$.
\end{lem}
\begin{proof}
According to Proposition \ref{prop:3.11}, we have $\rm{Ann}_{\it{Q}}(\it{I}) = \{\rm{0}\}$. For any $\mu \in rann_{A(Q)}(A_{Q}(I))$, we have $\rm{ad}\it{_{y}}\mu = \rm{0}$ for every $y \in I$. If $q \in Q$, we then have $0 = \rm{ad}\it{_{y}}\mu(q) = [\alpha(y), \mu(q)]$. This implies that $\mu(q) \in \rm{Ann}_{\it{Q}}(\it{I}) = \rm{0}$, and so $\mu = 0$.
\end{proof}
\begin{lem}\label{le:5.5}
Suppose that $(L, [\cdot, \cdot], \alpha)$ is a Hom-subalgebra of $(Q, [\cdot, \cdot], \alpha)$ and let $x_{1}, \cdots, x_{n}$, $y$ in $L$. Then we have, in $A(Q)$:
\[\rm{ad}\it{_{x_{1}}}\cdots\rm{ad}\it{_{x_{n}}}\rm{ad}\it{_{y}} = \rm{ad}\it{_{y}}\rm{ad}\it{_{x_{1}}}\cdots\rm{ad}\it{_{x_{n}}} + \sum_{i = 1}^{n}\rm{ad}\it{_{x_{1}}} \cdots\rm{ad}\it{_{\alpha([x_{i}, y])}}\cdots\rm{ad}\it{_{x_{n}}}.\]
In particular, if $I$ is a Hom-ideal of $L$ and $x_{1}, \cdots, x_{n} \in I$, then
\[\rm{ad}\it{_{x_{1}}}\cdots\rm{ad}\it{_{x_{n}}}\rm{ad}\it{_{y}} = \rm{ad}\it{_{y}}\rm{ad}\it{_{x_{1}}}\cdots\rm{ad}\it{_{x_{n}}} + \delta\]
where $\delta \in span\{\rm{ad}\it{_{z_{1}}}\cdots\rm{ad}\it{_{z_{n}}}\;|\;z_{i} \in I\}$.
\end{lem}
\begin{proof}
We'll prove it by induction on $n$.

When $n = 1$, we get
\[\rm{ad}\it{_{x}}\rm{ad}\it{_{y}} = \rm{ad}\it{_{\alpha([x, y])}} + \rm{ad}\it{_{y}}\rm{ad}\it{_{x}}\]
for any $x, y \in L$.

Suppose that $\rm{ad}\it{_{x_{1}}}\cdots\rm{ad}\it{_{x_{k}}}\rm{ad}\it{_{y}} = \rm{ad}\it{_{y}}\rm{ad}\it{_{x_{1}}}\cdots\rm{ad}\it{_{x_{k}}} + \sum_{i = 1}^{k}\rm{ad}\it{_{x_{1}}} \cdots\rm{ad}\it{_{\alpha([x_{i}, y])}}\cdots\rm{ad}\it{_{x_{k}}}$ for any $x_{1}, \cdots, x_{k}, y \in L$. Then when $n = k + 1$, for any $x_{1}, \cdots, x_{k}, x_{k + 1}, y \in L$, we get
\begin{align*}
&\rm{ad}\it{_{x_{1}}}\cdots\rm{ad}\it{_{x_{k}}}\rm{ad}\it{_{x_{k + 1}}}\rm{ad}\it{_{y}} = \rm{ad}\it{_{x_{1}}}\cdots\rm{ad}\it{_{x_{k}}}(\rm{ad}\it{_{y}}\rm{ad}\it{_{x_{k + 1}}} + \rm{ad}\it{_{\alpha([x_{k + 1}, y])}})\\
&= (\rm{ad}\it{_{x_{1}}}\cdots\rm{ad}\it{_{x_{k}}}\rm{ad}\it{_{y}})\rm{ad}\it{_{x_{k + 1}}} + \rm{ad}\it{_{x_{1}}}\cdots\rm{ad}\it{_{x_{k}}}\rm{ad}\it{_{\alpha([x_{k + 1}, y])}}\\
&= (\rm{ad}\it{_{y}}\rm{ad}\it{_{x_{1}}}\cdots\rm{ad}\it{_{x_{k}}} + \sum_{i = 1}^{k}\rm{ad}\it{_{x_{1}}} \cdots\rm{ad}\it{_{\alpha([x_{i}, y])}}\cdots\rm{ad}\it{_{x_{k}}})\rm{ad}\it{_{x_{k + 1}}} + \rm{ad}\it{_{x_{1}}}\cdots\rm{ad}\it{_{x_{k}}}\rm{ad}\it{_{\alpha([x_{k + 1}, y])}}\\
&= \rm{ad}\it{_{y}}\rm{ad}\it{_{x_{1}}}\cdots\rm{ad}\it{_{x_{k}}}\rm{ad}\it{_{x_{k + 1}}} + \sum_{i = 1}^{k + 1}\rm{ad}\it{_{x_{1}}} \cdots\rm{ad}\it{_{\alpha([x_{i}, y])}}\cdots\rm{ad}\it{_{x_{k + 1}}}.
\end{align*}
The proof is completed.
\end{proof}
Let $(L, [\cdot, \cdot], \alpha)$ be a Hom-subalgebra of $(Q, [\cdot, \cdot], \alpha)$. Denote by $A_{0}$ the associative subalgebra of $A(Q)$ whose elements are those $\mu$ in $A(Q)$ such that $\mu(L) \subseteq L$. We obviously have the containments:
\[A_{Q}(L) \subseteq A_{0} \subseteq A(Q).\]
\begin{lem}\label{le:5.6}
Suppose that $(L, [\cdot, \cdot], \alpha)$ is a Hom-subalgebra of $(Q, [\cdot, \cdot], \alpha)$ and $I$ a Hom-ideal of $(L, [\cdot, \cdot], \alpha)$. Let $q_{1}, \cdots, q_{n}$ in $Q$ such that $[\alpha(q_{i}), I] \subseteq L$ for every $i = 1, \cdots, n$. Then for $\mu = \rm{ad}\it{_{q_{1}}} \cdots \rm{ad}\it{_{q_{n}}}$ in $A(Q)$, we have that $\mu(\tilde{I})^{n} + (\tilde{I})^{n}\mu \subseteq A_{0}$ (where $(\tilde{I})^{n}$ denotes the $n$-th power of $\tilde{I}$ in the associative algebra $A_{Q}(L)$).
\end{lem}
\begin{proof}
According to Lemma \ref{le:5.2}, we have
\[(\tilde{I})^{n} = A_{Q}(L)(A_{Q}(I))^{n} + (A_{Q}(I))^{n} = (A_{Q}(I))^{n}A_{Q}(L) + (A_{Q}(I))^{n}.\]
Thus it's enough to prove that, for any $y = \rm{ad}\it{_{x_{1}}} \cdots \rm{ad}\it{_{x_{n}}}$ where $x_{i} \in I$, both $\mu y$ and $y \mu$ belong to $A_{0}$.

We'll use induction on $n$. For $n = 1$, we have $\rm{ad}\it{_{x}}\rm{ad}\it{_{q}} = \rm{ad}\it{_{\alpha([x, q])}} + \rm{ad}\it{_{q}}\rm{ad}\it{_{x}}$ and since $\alpha([x, q]) = [x, \alpha(q)] \in L$ we see that $\rm{ad}\it{_{\alpha([x, q])}} \in A_{\rm{0}}$. On the other hand, $\rm{ad}\it{_{q}}\rm{ad}\it{_{x}}(L) = [\alpha(q), [\alpha(x), L]] \subseteq [\alpha(q), I] \subseteq L$, and so $\rm{ad}\it{_{q}}\rm{ad}\it{_{x}} \in A_{\rm{0}}$.

Assume that the result is true for $n - 1$. Now, by Lemma \ref{le:5.5} we have
\begin{align*}\label{eq:*}
&\rm{ad}\it{_{x_{1}}}\cdots\rm{ad}\it{_{x_{n}}}\rm{ad}\it{_{q_{1}}}\cdots\rm{ad}\it{_{q_{n}}} = (\rm{ad}\it{_{q_{1}}}\rm{ad}\it{_{x_{1}}})\rm{ad}\it{_{x_{2}}}\cdots\rm{ad}\it{_{x_{n}}}\rm{ad}\it{_{q_{2}}}\cdots\rm{ad}\it{_{q_{n}}}\\
&+ \sum_{i = 1}^{n}\rm{ad}\it{_{x_{1}}} \cdots\rm{ad}\it{_{\alpha([x_{i}, q_{1}])}}\cdots\rm{ad}\it{_{x_{n}}}\rm{ad}\it{_{q_{2}}}\cdots\rm{ad}\it{_{q_{n}}}.\tag{*}
\end{align*}
The first summand on the right side belongs to $A_{0}$ because, as proved before, $\rm{ad}\it{_{q_{1}}}\rm{ad}\it{_{x_{1}}} \in A_{\rm{0}}$ and $\rm{ad}\it{_{x_{2}}}\cdots\rm{ad}\it{_{x_{n}}}\rm{ad}\it{_{q_{2}}}\cdots\rm{ad}\it{_{q_{n}}} \in A_{\rm{0}}$ by induction hypothesis. On the other hand, for each of the terms $\rm{ad}\it{_{x_{1}}} \cdots\rm{ad}\it{_{\alpha([x_{i}, q_{1}])}}\cdots\rm{ad}\it{_{x_{n}}}\rm{ad}\it{_{q_{2}}}\cdots\rm{ad}\it{_{q_{n}}}$ we have that $x_{i} \in I$ and $\alpha([x_{i}, q_{1}]) = [x_{i}, \alpha(q_{1})] \in L$. Using Lemma \ref{le:5.5}, we may write this as:
\[\rm{ad}\it{_{\alpha([x_{i}, q_{1}])}}\rm{ad}\it{_{x_{1}}} \cdots\rm{ad}\it{_{x_{i - 1}}}\rm{ad}\it{_{x_{i + 1}}}\cdots\rm{ad}\it{_{x_{n}}}\rm{ad}\it{_{q_{2}}}\cdots\rm{ad}\it{_{q_{n}}} + \delta \cdot \rm{ad}\it{_{x_{i + 1}}}\cdots\rm{ad}\it{_{x_{n}}}\rm{ad}\it{_{q_{2}}}\cdots\rm{ad}\it{_{q_{n}}},\]
where $\delta \in span\{\rm{ad}\it{_{z_{1}}}\cdots\rm{ad}\it{_{z_{i - 1}}}\;|\;z_{j} \in I\}$. The induction hypothesis applies again to show that this belongs to $A_{0}$. Hence, $y\mu \in A_{0}$.

If we continue to develop in the expression (\ref{eq:*}), we get, for some $\delta_{0} \in A_{0}$,
\[\rm{ad}\it{_{q_{1}}}\rm{ad}\it{_{q_{2}}}\rm{ad}\it{_{x_{1}}}\rm{ad}\it{_{x_{2}}}\cdots\rm{ad}\it{_{x_{n}}}\rm{ad}\it{_{q_{3}}}\cdots\rm{ad}\it{_{q_{n}}} + \sum_{i = 1}^{n}\rm{ad}\it{_{q_{1}}}\rm{ad}\it{_{x_{1}}} \cdots\rm{ad}\it{_{\alpha([x_{i}, q_{2}])}}\cdots\rm{ad}\it{_{x_{n}}}\rm{ad}\it{_{q_{3}}}\cdots\rm{ad}\it{_{q_{n}}} + \delta_{\rm{0}}.\]
Using Lemma \ref{le:5.5} we can write each term of the form
\[\rm{ad}\it{_{q_{1}}}\rm{ad}\it{_{x_{1}}} \cdots\rm{ad}\it{_{\alpha([x_{i}, q_{2}])}}\cdots\rm{ad}\it{_{x_{n}}}\rm{ad}\it{_{q_{3}}}\cdots\rm{ad}\it{_{q_{n}}}\]
as:
\[\rm{ad}\it{_{q_{1}}}\rm{ad}\it{_{\alpha([x_{i}, q_{2}])}}\rm{ad}\it{_{x_{1}}} \cdots\it{_{x_{i - 1}}}\rm{ad}\it{_{x_{i + 1}}}\cdots\rm{ad}\it{_{x_{n}}}\rm{ad}\it{_{q_{3}}}\cdots\rm{ad}\it{_{q_{n}}} + \rm{ad}\it{_{q_{1}}}\delta \cdot \rm{ad}\it{_{x_{i + 1}}}\cdots\rm{ad}\it{_{x_{n}}}\rm{ad}\it{_{q_{3}}}\cdots\rm{ad}\it{_{q_{n}}},\]
where $\delta \in span\{\rm{ad}\it{_{z_{1}}}\cdots\rm{ad}\it{_{z_{i - 1}}}\;|\;z_{j} \in I\}$. Notice that
\[\rm{ad}\it{_{q_{1}}}\rm{ad}\it{_{\alpha([x_{i}, q_{2}])}}\rm{ad}\it{_{x_{1}}} = \rm{ad}\it{_{q_{1}}}\rm{ad}\it{_{\alpha([\alpha([x_{i}, q_{2}]), x_{1}])}} + \rm{ad}\it{_{q_{1}}}\rm{ad}\it{_{x_{1}}}\rm{ad}\it{_{\alpha([x_{i}, q_{2}])}}.\]
Hence, using $\alpha([x_{i}, q_{2}]) \in L$ and $x_{i} \in I$, we see that the first summand above belongs to $A_{0}$. For the second summand, assuming that $\delta = \rm{ad}\it{_{z_{1}}}\cdots\rm{ad}\it{_{z_{i - 1}}}$ with $z_{j} \in I$, we have $(\rm{ad}\it{_{q_{1}}}\rm{ad}\it{_{z_{1}}})\rm{ad}\it{_{z_{2}}}\cdots\rm{ad}\it{_{z_{i - 1}}}\rm{ad}\it{_{x_{i + 1}}}\cdots\rm{ad}\it{_{x_{n}}}\rm{ad}\it{_{q_{3}}}\cdots\rm{ad}\it{_{q_{n}}}$, which is also an element of $A_{0}$. Continuing in this way, we find that
\[\rm{ad}\it{_{x_{1}}}\cdots\rm{ad}\it{_{x_{n}}}\rm{ad}\it{_{q_{1}}}\cdots\rm{ad}\it{_{q_{n}}} - \rm{ad}\it{_{q_{1}}}\cdots\rm{ad}\it{_{q_{n}}}\rm{ad}\it{_{x_{1}}}\cdots\rm{ad}\it{_{x_{n}}} \in A_{\rm{0}}\]
and by what we have just proved, we see that $\rm{ad}\it{_{q_{1}}}\cdots\rm{ad}\it{_{q_{n}}}\rm{ad}\it{_{x_{1}}}\cdots\rm{ad}\it{_{x_{n}}} \in A_{\rm{0}}$, as was to be shown.
\end{proof}
\begin{cor}\label{cor:5.7}
Suppose that $(L, [\cdot, \cdot], \alpha)$ is a Hom-subalgebra of $(Q, [\cdot, \cdot], \alpha)$ and $I$ a Hom-ideal of $(L, [\cdot, \cdot], \alpha)$. Let $q_{1}, \cdots, q_{n}$ in $Q$ such that $[\alpha(q_{i}), I] \subseteq L$ for every $i = 1, \cdots, n$. Then for $\mu = \rm{ad}\it{_{q_{1}}} \cdots \rm{ad}\it{_{q_{n}}}$ in $A(Q)$, we have that $\mu\tilde{I^{n}} \in A_{0}, \tilde{I^{n}}\mu \subseteq A_{0}$ (where $I^{1} = I$ and $I^{k} = [I^{k - 1}, \alpha(I)]$ for $k \geq 2$).
\end{cor}
\begin{proof}
It's straightforward to show that $I^{n}$ is a Hom-ideal of $L$ for each $n \geq 1$.

Claim that $\widetilde{I^{n}} \subseteq (\tilde{I})^{n}$. We'll show it by induction on $n$. When $n = 1$, it's obvious. Suppose that $\widetilde{I^{k}} \subseteq (\tilde{I})^{k}$. When $n = k + 1$, for any $y \in I^{k}, z \in I$,
\[\rm{ad}\it{_{[y, \alpha(z)]}} = \rm{ad}\it{_{\alpha([y, z])}} = \rm{ad}\it{_{y}}\rm{ad}\it{_{z}} - \rm{ad}\it{_{z}}\rm{ad}\it{_{y}} \in \tilde{I^{k}}\tilde{I} \subseteq (\tilde{I})^{k}\tilde{I} = (\tilde{I})^{k + 1}.\]
According to Lemma \ref{le:5.6}, we come to the conclusion.
\end{proof}
\begin{lem}\label{le:5.8}
Let $(L, [\cdot, \cdot], \alpha)$ be a semiprime Hom-Lie algebra. If $I$ is a Hom-ideal of $(L, [\cdot, \cdot], \alpha)$ with $\rm{Ann}_{\it{L}}(\it{I}) = \{\rm{0}\}$, then $\rm{Ann}_{\it{L}}(\it{I^{s}}) = \{\rm{0}\}$for any $s \geq 1$. Any finite intersection of ideals with zero $\alpha$-annihilator will also have zero $\alpha$-annihilator.
\end{lem}
\begin{proof}
We'll show it by induction on $s$. When $s = 1$, it's obvious. Suppose that $\rm{Ann}_{\it{L}}(\it{I^{k}}) = \{\rm{0}\}$, i.e., $I^{k}$ is essential. When $s = k + 1$, for any nonzero Hom-ideal $J$ of $L$, we have $I^{k} \cap J \neq \{0\}$ since $I^{k}$ is essential. So
\[\{0\} \neq [I^{k} \cap J, \alpha(I^{k} \cap J)] \subseteq [I^{k}, \alpha(I)] \cap J = I^{k + 1} \cap J,\]
which implies that $I^{k + 1}$ is essential and so $\rm{Ann}_{\it{L}}(\it{I^{k + 1}}) = \{\rm{0}\}$.

Similarly, we can show that any finite intersection of ideals with zero $\alpha$-annihilator will also have zero $\alpha$-annihilator by induction.
\end{proof}
\begin{prop}\label{prop:5.9}
Suppose that $(L, [\cdot, \cdot], \alpha)$ is a semiprime Hom-subalgebra of $(Q, [\cdot, \cdot], \alpha)$. Then the following conditions are equivalent:

(1) $Q$ is an algebra of quotients of $L$;

(2) $\rm{Ann}(\it{Q}) = \{\rm{0}\}$ and, if $\mu \in A(Q)\setminus\{0\}$, there exists a Hom-ideal $I$ of $L$ with $\rm{Ann}_{\it{L}}(\it{I}) = \{\rm{0}\}$ such that $\mu\tilde{I} \subseteq A_{0}$ and $\{0\} \neq \tilde{I}\mu \subseteq A_{0}$. If $\mu = \rm{ad}\it{_{q}}$, then we also have $\mu\tilde{I}(L) \neq \{0\}$.
\end{prop}
\begin{proof}
(2) $\Rightarrow$ (1) Let $q \in Q\setminus\{0\}$. Then $\mu = \rm{ad}\it{_{q}} \neq \rm{0}$ since $\rm{Ann}(\it{Q}) = \{\rm{0}\}$. Let $\tilde{I}$ be as in (2), so it satisfies $\mu\tilde{I}(L) \neq \{0\}$ and $\mu\tilde{I} \subseteq A_{0}$. Set
\[I_{0} := span\{\delta(x)\;|\;x \in L\; and\; \delta \in \tilde{I}\}.\]
Then $I_{0}$ is a Hom-ideal of $L$ such that $\rm{Ann}_{\it{L}}(\it{I}) = \{\rm{0}\}$. Indeed, for any $\delta \in \tilde{I}$, there exist a natural number $n$, elements $x_{1, i}, \cdots, x_{r_{i}, i} \in L$ and $y_{1, i}, \cdots, y_{s_{i}, i} \in I$ with $0 \leq r_{i} \in \mathbb{N}$ for all $i$ and $\emptyset \neq \{s_{1}, \cdots, s_{n}\} \subseteq \mathbb{N}$, such that
\[\delta = \sum_{i}^{n}\rm{ad}\it{_{x_{1, i}}} \cdots \rm{ad}\it{_{x_{r_{i}, i}}}\rm{ad}\it{_{y_{1, i}}} \cdots \rm{ad}\it{_{y_{s_{i}, i}}}.\]
For any $x, y \in L$, if $r_{i} = 0$ and $s_{i} = 1$, we get
\begin{align*}
&[y, \rm{ad}\it{_{y_{1, i}}}(x)] = [y, [\alpha(y_{1, i}), x]] = [\alpha(y_{1, i}), [y, x]] + [\alpha([y, y_{1, i}]), x]\\
&= \rm{ad}\it{_{y_{1, i}}}([y, x]) + \rm{ad}\it{_{[y, y_{1, i}]}}(x) \in I_{\rm{0}},
\end{align*}
if $r_{i} > 0$ and $s_{i} = 1$, we'll prove
\[[y, \rm{ad}\it{_{x_{1, i}}} \cdots \rm{ad}\it{_{x_{r_{i}, i}}}\rm{ad}\it{_{y_{1, i}}}(x)] \in I_{\rm{0}}\]
by induction on $r_{i}$.
When $r_{i} = 1$, we get
\begin{align*}
&[y, \rm{ad}\it{_{x_{1, i}}}\rm{ad}\it{_{y_{1, i}}}(x)] = [y, \rm{ad}\it{_{y_{1, i}}}\rm{ad}\it{_{x_{1, i}}}(x) + \rm{ad}\it{_{\alpha([x_{1, i}, y_{1, i}])}}(x)]\\
&= [y, \rm{ad}\it{_{y_{1, i}}}\rm{ad}\it{_{x_{1, i}}}(x)] + [y, \rm{ad}\it{_{\alpha([x_{1, i}, y_{1, i}])}}(x)]\\
&= [\alpha(y), \rm{ad}\it{_{y_{1, i}}}([x_{1, i}, x])] + [y, \rm{ad}\it{_{\alpha([x_{1, i}, y_{1, i}])}}(x)] = \rm{ad}\it{_{y}}\rm{ad}\it{_{y_{1, i}}}([x_{1, i}, x]) + [y, \rm{ad}\it{_{\alpha([x_{1, i}, y_{1, i}])}}(x)] \in I_{\rm{0}}.
\end{align*}
Suppose that when $r_{i} = k$, $[y, \rm{ad}\it{_{x_{1, i}}} \cdots \rm{ad}\it{_{x_{k, i}}}\rm{ad}\it{_{y_{1, i}}}(x)] \in I_{\rm{0}}$ for any $x_{1, i}, \cdots, x_{k, i} \in L$ and $y_{1, i} \in I$. When $r_{i} = k + 1$, we get
\begin{align*}
&[y, \rm{ad}\it{_{x_{1, i}}} \cdots \rm{ad}\it{_{x_{k + 1, i}}}\rm{ad}\it{_{y_{1, i}}}(x)]\\
&= \left[y, \left(\rm{ad}\it{_{y_{1, i}}}\rm{ad}\it{_{x_{1, i}}} \cdots \rm{ad}\it{_{x_{k + 1, i}}} + \sum_{j = 1}^{k + 1}\rm{ad}\it{_{x_{1, i}}} \cdots \rm{ad}\it{_{\alpha([x_{j, i}, y_{1, i}])}}\cdots\rm{ad}\it{_{x_{k + 1, i}}}\right)(x)\right]\\
&= [y, \rm{ad}\it{_{y_{1, i}}}\rm{ad}\it{_{x_{1, i}}} \cdots \rm{ad}\it{_{x_{k + 1, i}}}(x)] + \left[y, \sum_{j = 1}^{k + 1}\rm{ad}\it{_{x_{1, i}}} \cdots \rm{ad}\it{_{\alpha([x_{j, i}, y_{1, i}])}}\cdots\rm{ad}\it{_{x_{k + 1, i}}}(x)\right]\\
&= [\alpha(y), \rm{ad}\it{_{y_{1, i}}}\rm{ad}\it{_{x_{1, i}}} \cdots \rm{ad}\it{_{x_{k, i}}}([x_{k + 1, i}, x])] + \sum_{j = 1}^{k}[\alpha(y), \rm{ad}\it{_{x_{1, i}}} \cdots \rm{ad}\it{_{\alpha([x_{j, i}, y_{1, i}])}}\cdots\rm{ad}\it{_{x_{k, i}}}([x_{k + 1, i}, x])]\\
&+ [y, \rm{ad}\it{_{x_{1, i}}} \cdots \rm{ad}\it{_{x_{k, i}}}\rm{ad}\it{_{\alpha([x_{k + 1, i}, y_{1, i}])}}(x)]\\
&= \rm{ad}\it{_{y}}\rm{ad}\it{_{y_{1, i}}}\rm{ad}\it{_{x_{1, i}}} \cdots \rm{ad}\it{_{x_{k, i}}}([x_{k + 1, i}, x]) + \sum_{j = 1}^{k}\rm{ad}\it{_{y}}\rm{ad}\it{_{x_{1, i}}} \cdots \rm{ad}\it{_{\alpha([x_{j, i}, y_{1, i}])}}\cdots\rm{ad}\it{_{x_{k, i}}}([x_{k + 1, i}, x])\\
&+ [y, \rm{ad}\it{_{x_{1, i}}} \cdots \rm{ad}\it{_{x_{k, i}}}\rm{ad}\it{_{\alpha([x_{k + 1, i}, y_{1, i}])}}(x)]\\
&= \rm{ad}\it{_{y}}\rm{ad}\it{_{x_{1, i}}} \cdots \rm{ad}\it{_{x_{k, i}}}\rm{ad}\it{_{y_{1, i}}}([x_{k + 1, i}, x]) - \sum_{j = 1}^{k}\rm{ad}\it{_{y}}\rm{ad}\it{_{x_{1, i}}} \cdots \rm{ad}\it{_{\alpha([x_{j, i}, y_{1, i}])}}\cdots\rm{ad}\it{_{x_{k, i}}}([x_{k + 1, i}, x])\\
&+ \sum_{j = 1}^{k}\rm{ad}\it{_{y}}\rm{ad}\it{_{x_{1, i}}} \cdots \rm{ad}\it{_{\alpha([x_{j, i}, y_{1, i}])}}\cdots\rm{ad}\it{_{x_{k, i}}}([x_{k + 1, i}, x])  + [y, \rm{ad}\it{_{x_{1, i}}} \cdots \rm{ad}\it{_{x_{k, i}}}\rm{ad}\it{_{\alpha([x_{k + 1, i}, y_{1, i}])}}(x)]\\
&= \rm{ad}\it{_{y}}\rm{ad}\it{_{x_{1, i}}} \cdots \rm{ad}\it{_{x_{k, i}}}\rm{ad}\it{_{y_{1, i}}}([x_{k + 1, i}, x]) + [y, \rm{ad}\it{_{x_{1, i}}} \cdots \rm{ad}\it{_{x_{k, i}}}\rm{ad}\it{_{\alpha([x_{k + 1, i}, y_{1, i}])}}(x)] \in I_{\rm{0}},
\end{align*}
if $s_{i} > 1$, we have
\begin{align*}
&[y, \rm{ad}\it{_{x_{1, i}}} \cdots \rm{ad}\it{_{x_{r_{i}, i}}}\rm{ad}\it{_{y_{1, i}}} \cdots \rm{ad}\it{_{y_{s_{i}, i}}}(x)] = [y, \rm{ad}\it{_{x_{1, i}}} \cdots \rm{ad}\it{_{x_{r_{i}, i}}}\rm{ad}\it{_{y_{1, i}}} \cdots \rm{ad}\it{_{y_{s_{i} - 1, i}}}([\alpha(y_{s_{i}, i}), x])]\\
&= [\alpha(y), \rm{ad}\it{_{x_{1, i}}} \cdots \rm{ad}\it{_{x_{r_{i}, i}}}\rm{ad}\it{_{y_{1, i}}} \cdots \rm{ad}\it{_{y_{s_{i} - 1, i}}}([y_{s_{i}, i}, x])]\\
&= \rm{ad}\it{_{y}}\rm{ad}\it{_{x_{1, i}}} \cdots \rm{ad}\it{_{x_{r_{i}, i}}}\rm{ad}\it{_{y_{1, i}}} \cdots \rm{ad}\it{_{y_{s_{i} - 1, i}}}([y_{s_{i}, i}, x]) \in I_{\rm{0}}.
\end{align*}
Therefore we get $[y, \delta(x)] \in I_{0}$, i.e., $[L, I_{0}] \subseteq I_{0}$. It's straightforward to show that $\alpha(I_{0}) \subseteq I_{0}$.

If now $[x, \alpha(I_{0})] = \{0\}$ for $x \in L$, then $\rm{ad}\it{_{x}}\tilde{I}(L) = \{\rm{0}\}$. In particular, for $y, z \in I$, we have $\rm{ad}\it{_{x}}\rm{ad}\it{_{y}}(z) = \rm{0}$ and so $x \in \rm{Ann}_{\it{L}}(\it{I^{2}}) = \{\rm{0}\}$, which is zero by Lemma \ref{le:5.8}.

Finally, $\{0\} \neq [\alpha(q), I_{0}] = \rm{ad}\it{_{q}}\tilde{I}(L) \subseteq A_{\rm{0}}(L) \subseteq L$, which implies that $Q$ is ideally absorbed into $L$. According to Theorem \ref{thm:3.10}, $Q$ is an algebra of quotients of $L$.

(1) $\Rightarrow$ (2) Since $Q$ is an algebra of quotients of $L$, we have $\rm{Ann}(\it{Q}) = \{\rm{0}\}$. Next let $\mu = \sum_{i \geq 1}\rm{ad}\it{_{q_{i, 1}}} \cdots \rm{ad}\it{_{q_{i, r_{i}}}} \in A(Q)\setminus\{\rm{0}\}$. We may of course assume that all $q_{i, j}$ are nonzero elements in $Q$. Set $s = \sum_{i \geq 1}r_{i}$. As $Q$ is an algebra of quotients of $L$, there exists, for every $i$ and $j$, a Hom-ideal $J_{i, j}$ of $L$ such that $\rm{Ann}_{\it{L}}(\it{J_{i, j}}) = \{\rm{0}\}$, and $\{0\} \neq [\alpha(q_{i, j}), J_{i, j}] \subseteq L$. By Lemma \ref{le:5.8}, the Hom-ideal $J = \cap_{i, j}J_{i, j}$ and hence also $I = J^{s}$ will have zero $\alpha$-annihilator in $L$. Moreover, $[\alpha(q_{i, j}), I] \subseteq [\alpha(q_{i, j}), J_{i, j}] \subseteq L$. According to Corollary \ref{cor:5.7} and taking into account that $J^{s} \subseteq J^{r_{i}}$ for every $i$, we have $\mu\tilde{I} \subseteq A_{0}$ and $\tilde{I}\mu \subseteq A_{0}$.

If $\tilde{I}\mu = \{0\}$, then $\mu \in rann_{A(Q)}(A_{Q}(I))$ which is zero by Lemma \ref{le:5.4}.

Finally, suppose that $\mu = \rm{ad}\it{_{q}}$ where $q \in A(Q)\setminus\{0\}$. Note that $Q$ is also an algebra of quotients of $I^{2}$, this implies that there exist elements $y, z \in I$ such that $0 \neq [\alpha(q), [y, \alpha(z)]]$. But $[\alpha(q), [y, \alpha(z)]] = \rm{ad}\it{_{q}}\rm{ad}\it{_{y}}(z) \in \rm{ad}\it{_{q}}\tilde{I}(L)$.
\end{proof}
Recall that an associative algebra $S$ is said to be a left quotient algebra of a subalgebra $A$ if whenever $p$ and $q \in S$, with $p \neq 0$, there exists $x$ in $A$ such that $xp \neq 0$ and $xq \in A$. An associative algebra $A$ has a left quotient algebra if and only if it has no total right zero divisors different from zero. (Here, an element $x$ in $A$ is said to be a total right zero divisor if $Ax = \{0\}$.)
\begin{lem}\label{le:5.10}\cite{S1}
Let $A$ be a subalgebra of an associative algebra $Q$. Then $Q$ is a left quotient algebra of $L$ if and only if for every nonzero element $q \in Q$ there exists a left ideal $I$ of $A$ with $rann_{A}(I) = \{0\}$ such that $\{0\} \neq Iq \subseteq A$.
\end{lem}
\begin{thm}\label{thm:5.11}
Suppose that $(L, [\cdot, \cdot], \alpha)$ is a semiprime Hom-subalgebra of $(Q, [\cdot, \cdot], \alpha)$. Moreover, suppose that $Q$ is an algebra of quotients of $L$. Then $A(Q)$ is a left quotient algebra of $A_{0}$.
\end{thm}
\begin{proof}
Let $\mu \in A(Q)\setminus\{0\}$ and let $I$ be a Hom-ideal of $L$ satisfying condition (2) in Proposition \ref{prop:5.9}. Set $J = A_{0}\tilde{I} + \tilde{I}$, a left ideal of $A_{0}$ that satisfies $\{0\} \neq J\mu \subseteq A_{0}$ (because $\{0\} \neq \tilde{I}\mu \subseteq A_{0}$).

Since also $\rm{Ann}_{\it{L}}(\it{I}) = \{\rm{0}\}$, we obtain from Lemma \ref{le:5.4} that $rann_{A(Q)}(A_{Q}(I)) = \{0\}$. This, together with the fact that $\tilde{I}$ contains $A_{Q}(I)$, implies that $rann_{A(Q)}(\tilde{I}) = \{0\}$. Since $rann_{A_{0}}(\tilde{I}) \subseteq rann_{A(Q)}(\tilde{I})$ we get that also $rann_{A_{0}}(\tilde{I}) = \{0\}$. Hence we get $rann_{A_{0}}(J) = \{0\}$ since $\tilde{I} \subseteq J$. This concludes the proof according to Lemma \ref{le:5.10}.
\end{proof}
\begin{re}\label{re:5.12}
As established in the proof of the previous result, if $(L, [\cdot, \cdot], \alpha)$ is a Hom-subalgebra of $(Q, [\cdot, \cdot], \alpha)$ and if $(L, [\cdot, \cdot], \alpha)$ is semiprime and $I$ is a Hom-ideal of $(L, [\cdot, \cdot], \alpha)$ with $\rm{Ann}_{\it{L}}(\it{I}) = \{\rm{0}\}$, then $A_{0}\tilde{I} + \tilde{I}$ is a left ideal of $A_{0}$ with zero right annihilator.
\end{re}
\section{Algebras of quotients of Hom-Lie algebras with dense extensions}\label{se:6}
In this section, we mainly study algebras of quotients of Hom-Lie algebras via their dense extensions which is introduced by Cabrera in \cite{C1}. We show that dense extension can be lifted from a Hom-Lie algebra to its Hom-ideals if the extension is also an algebra of quotients. For any Hom-Lie algebra $(L, [\cdot, \cdot], \alpha)$, let $M(L)$ be the associative algebra generated by the identity map together with inner derivations of $L$. Moreover, $M(L)$ is nothing but the unitization of $A(L)$.

Following \cite{C1}, given an extension of Hom-Lie algebras $L \subseteq Q$, the annihilator of $L$ in $M(Q)$ is defined by
\[L^{ann} := \{\mu \in M(Q)\; |\; \mu(x) = 0,\; x \in L\}.\]
If $L^{ann} = \{0\}$, then $L$ is said to be a dense Hom-subalgebra of $Q$, and we will say that $L \subseteq Q$ is a dense extension of Hom-Lie algebras.
\begin{lem}\label{le:6.1}
Suppose that $(L, [\cdot, \cdot], \alpha)$ is a Hom-subalgebra of $(Q, [\cdot, \cdot], \alpha)$ with $\rm{Ann}(\it{Q}) = \{\rm{0}\}$. Then the following conditions are equivalent:
\begin{enumerate}[(1)]
\item $L$ is a dense Hom-subalgebra of $Q$.
\item If $\mu(L) = \{0\}$ for some $\mu$ in $A(Q)$, then $\mu = 0$.
\end{enumerate}
\end{lem}
\begin{proof}
Clearly, (1) implies (2). Conversely, suppose that $\mu \in M(Q)$ satisfies $\mu(L) = \{0\}$. If $\mu(p) \neq 0$ for some $p$ in $Q$, then use $\rm{Ann}(\it{Q}) = \{\rm{0}\}$ to find a nonzero elements $q$ in $Q$ satisfying $\rm{ad}\it{_{q}}\mu(p) \neq \rm{0}$. But then $\rm{ad}\it{_{q}}\mu(L) = \{\rm{0}\}$ and since $A(Q)$ is a two-sided ideal of $M(Q)$, we have that $\rm{ad}\it{_{q}}\mu \in A(Q)$. Hence, condition (2) yields $\rm{ad}\it{_{q}}\mu = \rm{0}$, a contradiction.
\end{proof}
Following Cabrera and Mohammed in \cite{CM}, we say that a Hom-Lie algebra $(L, [\cdot, \cdot], \alpha)$ is multiplicatively semiprime whenever $L$ and its multiplication algebra $M(L)$ are semiprime. Observe that in this situation, and if $L$ is a Hom-Lie algebra, then $A(L)$, being an ideal of $M(L)$, will also be a semiprime algebra.
\begin{lem}\label{le:6.2}
Suppose that $(L, [\cdot, \cdot], \alpha)$ is a dense Hom-subalgebra of $(Q, [\cdot, \cdot], \alpha)$ and $I$ a Hom-ideal of $L$. If $\mu$ is an element of $M(Q)$ such that $\mu(I) = \{0\}$, then $\mu(M(Q)(I)) = \{0\}$.
\end{lem}
\begin{proof}
Consider the set
\[\mathcal{H} = \{\varphi \in M(Q)\;|\;\mu\omega\varphi(I) = \{0\}\;\rm{for\;each}\;\omega \in \it{M(Q)}\;\rm{such\;that}\;\mu\omega(\it{I}) = \{\rm{0}\}\}.\]
It's clear that $\mathcal{H}$ is a subalgebra of $M(Q)$. Indeed, for any $\varphi, \psi \in \mathcal{H}$ and each $\omega \in M(Q)$ such that $\mu\omega(I) = \{0\}$, we get that $\mu\omega\varphi\psi(I) = \{0\}$ since $\omega\varphi \in M(Q)$ satisfying $\mu\omega\varphi(I) = \{0\}$ and $\psi \in \mathcal{H}$, which implies that $\varphi\psi \in \mathcal{H}$.

Furthermore, for $\omega \in M(Q)$ such that $\mu\omega(I) = \{0\}$, $x \in I$ and $y \in L$ we see that
\[0 = \mu\omega([\alpha(x), y]) = \mu\omega\rm{ad}\it{_{x}}(y).\]
Therefore $\mu\omega\rm{ad}\it{_{x}}$ lies in $L^{ann}$, and so it is equal to zero. Consequently for each $z \in Q$, it can be seen that
\[0 = \mu\omega\rm{ad}\it{_{x}}(z) = \mu\omega([\alpha(x), z]) = -\mu\omega([\alpha(z), x]) = -\mu\omega\rm{ad}\it{_{z}}(x).\]
From this it follows that $\mu\omega\rm{ad}\it{_{z}}(I) = \{\rm{0}\}$, and hence $\rm{ad}\it{_{z}}$ belongs to $\mathcal{H}$. Thus $\mathcal{H} = M(Q)$. Finally the statement follows directly from the definition of $\mathcal{H}$ and from the fact that, by assumption, $\mu\rm{id}_{\it{Q}}(\it{I}) = \{\rm{0}\}$.
\end{proof}
\begin{cor}\label{cor:6.3}
Suppose that $(L, [\cdot, \cdot], \alpha)$ is a dense Hom-subalgebra of $(Q, [\cdot, \cdot], \alpha)$ and $I, J$ two Hom-ideals of $L$. If $[\alpha(I), J] = \{0\}$, then $[M(Q)(I), \alpha(M(Q)(J))] = \{0\}$.
\end{cor}
\begin{proof}
First, apply Lemma \ref{le:6.2} to the Hom-ideal $J$ taking $\mu = \rm{ad}\it{_{x}}$ for any $x \in I$ to obtain the result $[\alpha(I), M(Q)(J)] = \{0\}$. Now, to conclude the proof, apply once more Lemma \ref{le:6.2} to the Hom-ideal $I$ taking $\mu = \rm{ad}\it{_{z}}$ where $z \in M(Q)(J)$.
\end{proof}
\begin{cor}\label{cor:6.4}
Suppose that $(L, [\cdot, \cdot], \alpha)$ is a dense Hom-subalgebra of $(Q, [\cdot, \cdot], \alpha)$. If $(Q, [\cdot, \cdot], \alpha)$ is semiprime, then $(L, [\cdot, \cdot], \alpha)$ is also semiprime.
\end{cor}
\begin{proof}
Otherwise, there exists a nonzero Hom-ideal $I$ such that $[I, \alpha(I)] = \{0\}$. By Corollary \ref{cor:6.3}, $[M(Q)(I), \alpha(M(Q)(I))] = \{0\}$ where $M(Q)(I)$ is a Hom-ideal of $Q$, contradiction.
\end{proof}
\begin{lem}\label{le:6.5}
Suppose that $(L, [\cdot, \cdot], \alpha)$ is a dense Hom-subalgebra of $(Q, [\cdot, \cdot], \alpha)$. If $Q$ is multiplicatively semiprime, then $L$ is also multiplicatively semiprime.
\end{lem}
\begin{proof}
It's sufficient to show that $M(L)$ is semiprime by Corollary \ref{cor:6.4}. Suppose that $\mu$ in $M(L)$ satisfies $\mu M(L)\mu = 0$. Then $\mu(I) = \{0\}$ where $I$ denotes the Hom-ideal of $L$ generated by $\mu(L)$. Now, choose an element of $M(Q)$, say $\hat{\mu}$, satisfying $\hat{\mu}|_{L} = \mu|_{L}$. Hence $\hat{\mu}(I) = \{0\}$. By Lemma \ref{le:6.2}, $\hat{\mu}(M(Q)(I)) = \{0\}$, and so $\hat{\mu}M(Q)\hat{\mu}(L) = \{0\}$. Consequently we have $\hat{\mu}M(Q)\hat{\mu} = 0$ since $L \subseteq Q$ is dense. Note that $M(Q)$ is semiprime, we get $\hat{\mu} = 0$, and so $\mu = 0$.
\end{proof}
\begin{re}
Given a dense extension of Hom-Lie algebras $L \subseteq Q$, where $Q$ is multiplicatively semiprime, we have that $L$ is also multiplicatively semiprime from Lemma \ref{le:6.5}.
\end{re}
\begin{prop}\label{prop:6.7}
Suppose that $(L, [\cdot, \cdot], \alpha)$ is a dense Hom-subalgebra of $(Q, [\cdot, \cdot], \alpha)$ and $Q$ a multiplicative semiprime algebra of quotients of $L$. Then $I \subseteq Q$ is a dense extension for every essential Hom-ideal of $I$ of $L$.
\end{prop}
\begin{proof}
We first observe that $\rm{Ann}(\it{Q}) = \{\rm{0}\}$ since $Q$ is an algebra of quotients of $L$ and hence Lemma \ref{le:6.1} applies. Thus, let $\mu$ be in $A(Q)$ such that $\mu(I) = \{0\}$, and by way of contradiction assume that $\mu \neq 0$. According to Theorem \ref{thm:5.11}, $A(Q)$ is a left quotients algebra of $A_{0}$ and hence there exists $\lambda$ in $A_{0}$ such that $0 \neq \lambda\mu \in A_{0}$. Since the extension $L \subseteq Q$ is dense, $\lambda\mu(L) \neq \{0\}$, and since $\rm{Ann}_{\it{L}}(\it{I}) = \{\rm{0}\}$, there exists a nonzero element $y$ in $I$ such that $\rm{ad}\it{_{y}}\lambda\mu(L) \neq \{\rm{0}\}$. Using now that $A(Q)$ has no total right zero divisor according to Lemma \ref{le:5.1}, we get that $A(Q)\rm{ad}\it{_{y}}\lambda\mu \neq \rm{0}$, and this, coupled with the semiprimeness of $A(Q)$, implying that $A(Q)\rm{ad}\it{_{y}}\lambda\mu A(Q)\rm{ad}\it{_{y}}\lambda\mu \neq \rm{0}$. A second application of the fact that $L \subseteq Q$ is a dense extension yields $A(Q)\rm{ad}\it{_{y}}\lambda\mu A(Q)\rm{ad}\it{_{y}}\lambda\mu(L) \neq \{\rm{0}\}$. However, $\mu(I) = \{0\}$ by assumption and thus implies that $\mu M(Q)(I) = \{0\}$. But this is a contradiction, because of the containments
\[\mu A(Q)\rm{ad}\it{_{y}}\lambda\mu(L) \subseteq \mu A(Q)([\alpha(I), L]) \subseteq \mu A(Q)(I) \subseteq \mu M (Q)(I) = \{\rm{0}\}.\]
This completes the proof.
\end{proof}
\begin{cor}\label{cor:6.8}
Suppose that $(L, [\cdot, \cdot], \alpha)$ is a dense Hom-subalgebra of $(Q, [\cdot, \cdot], \alpha)$ and $Q$ a multiplicative semiprime algebra of quotients of $L$. Then for every essential Hom-ideal $I$ of $L$, $lann_{A(Q)}(\tilde{I}) = \{0\}$.
\end{cor}
\begin{proof}
Let $\mu \in lann_{A(Q)}(\tilde{I})$. Then, if $y \in I$, we have $\mu\rm{ad}\it{_{y}}(I) = \{\rm{0}\}$. This implies that $\mu(I^{2}) = \{0\}$. By Proposition \ref{prop:6.7} applied to the essential Hom-ideal $I^{2}$ of $L$, the extension $I^{2} \subseteq Q$ is dense, so $\mu = 0$.
\end{proof}

\end{document}